\documentclass[11pt,reqno]{amsart}

\usepackage{amsmath, amsthm, amssymb}
\usepackage{amsfonts}

\usepackage{hyperref}

\usepackage[ansinew]{inputenc}
\usepackage[dvips]{epsfig}
\usepackage{graphicx}
\usepackage[english]{babel}

\usepackage{thmtools}
\theoremstyle{plain}
\declaretheorem[title=Theorem, parent=section]{theorem}
\declaretheorem[title=Lemma,sibling=theorem]{lemma}
\declaretheorem[title=Proposition,sibling=theorem]{proposition}
\declaretheorem[title=Corollary,sibling=theorem]{corollary}

\theoremstyle{definition}
\declaretheorem[title=Definition,sibling=theorem]{definition}
\declaretheorem[title=Remark,sibling=theorem]{remark}
\declaretheorem[title=Remark, numbered=no]{remark*}

\declaretheorem[title=Assumption, numbered=no]{assumption*}

\numberwithin{equation}{section}

\usepackage[backgroundcolor=white, bordercolor=blue,
linecolor=blue]{todonotes}

\usepackage{dsfont}
\usepackage{bbm}

\newcommand{\R}{\mathbb{R}}

\newcommand{\cE}{\mathcal{E}}

\newcommand{\eps}{\varepsilon}

\newcommand{\1}{\mathbbm{1}}

\DeclareMathOperator{\dist}{dist}

\renewcommand{\d}{\textnormal{\,d}}

\newcommand{\average}{{\mathchoice {\kern1ex\vcenter{\hrule height.4pt
width 6pt depth0pt} \kern-9.7pt} {\kern1ex\vcenter{\hrule
height.4pt width 4.3pt depth0pt} \kern-7pt} {} {} }}
\newcommand{\dashint}{\average\int}

\begin{document}

\allowdisplaybreaks

 \title[Regularity for the Boltzmann equation under pressure and moment bounds]{Regularity for the Boltzmann equation conditional to pressure and moment bounds}

\author{Xavier Fern\'andez-Real}

\author{Xavier Ros-Oton}

\author{Marvin Weidner}

\address{EPFL SB, Station 8, CH-1015 Lausanne, Switzerland}
\email{xavier.fernandez-real@epfl.ch}

\address{ICREA, Pg. Llu\'is Companys 23, 08010 Barcelona, Spain \& Universitat de Barcelona, Departament de Matem\`atiques i Inform\`atica, Gran Via de les Corts Catalanes 585, 08007 Barcelona, Spain \& Centre de Recerca Matem\`atica, Barcelona, Spain}
\email{xros@icrea.cat}

\address{Departament de Matem\`atiques i Inform\`atica, Universitat de Barcelona, Gran Via de les Corts Catalanes 585, 08007 Barcelona, Spain}
\email{mweidner@ub.edu}



\begin{abstract}
We prove that solutions to the Boltzmann equation without cut-off satisfying pointwise bounds on some observables (mass, pressure, and suitable moments) enjoy a uniform bound in~$L^\infty$ in the case of hard potentials. As a consequence, we derive $C^{\infty}$ estimates and decay estimates for all derivatives, conditional to these macroscopic bounds. Our $L^\infty$ estimates are uniform in the limit $s \nearrow 1$ and hence we recover the same results also for the Landau equation.
\end{abstract}

\allowdisplaybreaks

\maketitle
\section{Introduction}

\subsection{The Boltzmann equation}

The Boltzmann equation is one of the fundamental equations of statistical mechanics.
It models the evolution of a gas (or any system made up of a large number of particles), and it was derived by Boltzmann and Maxwell in the 19th century.

The unknown in Boltzmann's equation is a time-dependent probability density $f(t,x,v)$ which keeps track of the ``number'' of particles that at time $t$ and point $x$ have velocity $v$,
\begin{align}\label{eq:Boltzmann}
\partial_t f + v \cdot \nabla_x f = Q(f,f) ~~ \text{ in } (0,\infty) \times \R^n \times \R^n,
\end{align}
where $Q(f,f)$ is the so-called Boltzmann collision operator, and $n \ge 2$.

The Boltzmann collision operator acts only on the velocity variable $v$, and is of the form
\begin{align*}
Q(f,g)(v) = \int_{\R^n \times \mathbb{S}^{n-1}} \big(f(v_{\ast}') g(v') - f(v_{\ast})g(v) \big) B(|v_{\ast} - v|,\cos \theta) \d \sigma \d v_{\ast},
\end{align*}
where $\cos \theta = \frac{v-v_{\ast}}{|v-v_{\ast}|} \cdot \sigma$,  $B$ is the so-called collision kernel, and $v'$ and $ v_{\ast}'$ are the post-collisional velocities given (under elastic collisions) by
\begin{align}
\label{eq:v-v-prime}
v' = \frac{v + v_{\ast}}{2} + \frac{|v-v_{\ast}|}{2} \sigma, \qquad v_{\ast}' = \frac{v + v_{\ast}}{2} - \frac{|v - v_{\ast}|}{2} \sigma.
\end{align} 

The exact form of the collision kernel $B$ depends on the microscopic interaction that we assume between the particles: they interact with each other via a (repulsive) potential $\phi$, most typically with an inverse-power law $\phi(r)=1/r^p$, with $p>1$.
Under these assumptions, we have
\begin{equation}
\label{eq:Bdef}
B(r,\cos \theta) = r^{\gamma} b(\cos \theta), \qquad b(\cos \theta) \asymp |\sin (\theta/2)|^{-(n-1)-2s},
\end{equation}
for some $s \in (0,1)$ and $\gamma>-n$ (see \eqref{s-to-1-kernel} as well).
In the most physically relevant case, $n=3$ and inverse-power law potentials, we actually have $s=\frac1p$ and $\gamma=1-\frac{4}{p}$.
Still, for the sake of generality, the Boltzmann equation is typically studied for general independent parameters $s\in (0,1)$ and $\gamma>-n$.

An important distinction arises often related to the ``strength'' of the repulsive potential $\phi$: when $\gamma>0$ we talk about \emph{hard potentials}, while the case $\gamma\leq0$ is called \emph{soft potentials}.

The limiting case $p\to\infty$ corresponds to hard spheres (in which the collision kernel is not singular anymore, since $s\to0$), while the case $p\to 1$ corresponds to the Coulomb interaction (in which the Boltzmann equation becomes the Landau equation, and $s\to1$).

An important feature of the Boltzmann equation is that it keeps track of macroscopic information (``observables''), but also microscopic variables, which describe the state of the particles at a given time.
All macroscopic observables can be expressed in terms of microscopic averages, i.e., integrals of the form $\int f(t,x,v)\varphi(v) \d v$.
In particular, at any time $t$ and any given point $x$, we have the following observables
\begin{alignat}{2}
 \rho(t,x) &= \int_{\R^n} f(t,x,v) \d v  &&\qquad \textrm{(mass density)} \\
\bar{v}(t,x)&=\frac{1}{\rho} \int_{\R^n} f(t,x,v)v \d v  &&\qquad \textrm{(mean velocity)} \\
\mathbb P(t,x)&= \int_{\R^n} f(t,x,v)\,(v-\bar{v})\otimes (v-\bar{v}) \d v  &&\qquad \textrm{(pressure tensor)} \\ \label{eq:temp}
T(t,x)&= \frac{1}{n\rho} {\rm tr\,} \mathbb P=\frac{1}{n\rho}\int_{\R^n} f(t,x,v)|v-\bar{v}|^2 \d v  &&\qquad \textrm{(temperature)} \\
E(t,x)&= \frac12\rho|\bar{v}|^2+\frac{n}{2}\rho T=\frac{1}{2}\int_{\R^n} f(t,x,v)|v|^2 \d v  &&\qquad \textrm{(energy density)};
\end{alignat}
see, e.g. the survey \cite{Vi02} for more details.

Of course, the equation can also be posed in a bounded domain $\Omega\subset \R^n$ with appropriate boundary conditions (see, e.g., \cite{OuSi24}), however in this paper we focus for simplicity on the case $\Omega=\R^n$.

\subsection{Regularity for the Boltzmann equation}

One of the most important and famous mathematical results for the Boltzmann equation is the {convergence to equilibrium} for smooth solutions, established by Desvilletes and Villani in~\cite{DeVi05}.
The result may be informally summarized as follows:

\vspace{2mm}

\emph{Let $f$ be any solution to the Boltzmann equation, with appropriate decay for large velocities, such that $f$ stays in $C^\infty$ in all variables, uniformly for all $t>0$. 
\\ Then, it converges to equilibrium as $t\to\infty$ faster than any algebraic rate $O(t^{-k})$, $k\in \mathbb N$.}

\vspace{2mm}

This is one of the main two results for which Villani received the Fields Medal in 2010---see \cite[Theorem 2]{DeVi05} for a precise statement.

Their result hence reduces the problem of convergence to equilibrium to the problem of establishing a priori bounds on moments and $C^k$ norms, uniformly in time.
Furthermore, they conjectured that one should be able to establish these bounds, \emph{conditionally to global in time a priori estimates on the hydrodynamic fields $\rho$, $\bar{v}$, and $T$.}

This was essentially the program carried out by Imbert and Silvestre (and Mouhot) in the last years \cite{ImSi22,Sil16,IMS20,ImSi20b,ImSi21} (see also the survey \cite{ImSi20}), who established the uniform $C^\infty$ regularity and decay for (periodic in~$x$) solutions to the Boltzmann equation \eqref{eq:Boltzmann}, under the assumption that the mass density $\rho$ and energy $E$ satisfy
\begin{align}\label{eq:mass}
0<m_0 \le \rho(t,x) &:= \int_{\R^n} f(t,x,v) \d v \le M_0,\\
\label{eq:energy}
E(t,x) &:= \frac{1}{2}\int_{\R^n} f(t,x,v) |v|^2 \d v \le E_0,
\end{align}
and also that the \emph{entropy density} is controlled
\begin{equation} \label{eq:entropy}
h(t,x) := \int_{\R^n} f \log f (t,x,v) \d v \leq H_0 \qquad \textrm{(entropy density)}.
\end{equation}

Their main result, which holds for $\gamma + 2s \in [0,2]$, can be informally summarized as follows:

\vspace{2mm}

\emph{Let $f$ be any solution to the Boltzmann equation satisfying \eqref{eq:mass}-\eqref{eq:energy}-\eqref{eq:entropy} uniformly in $t$, $x$. \\
Then, $f$ stays in $C^\infty$ in all variables (with fast decay as $v\to \infty$), uniformly for all $t>0$. }

\vspace{2mm}

Their results apply to strong solutions to the Boltzmann equation:

\begin{definition}
\label{def:solution}
A function $f : (0,T) \times \R^n \times \R^n \to \R$ is said to be a solution to the Boltzmann equation \eqref{eq:Boltzmann} if $0\leq f \in C^{\infty}((0,T) \times \R^n \times \R^n)$ satisfies \eqref{eq:Boltzmann} in the pointwise sense for all $(t,x,v)\in (0,T)\times \R^n\times \R^n$. 
Moreover, we assume that $f$ is periodic in $x$, that for any $q > 0$ we have
\begin{align*}
\lim_{|v| \to \infty} \frac{f(t,x,v)}{|v|^q} = 0
\end{align*}
locally uniformly in $(t,x)$, and in addition that for every $(t,x)$ it holds $\int_{\R^n} |D^2_v f| (1 + |v|)^{\gamma + 2s} \d v < \infty$.
\end{definition}

We will use the same notion of solution in this paper.

\subsection{Our results}

Notice that the entropy assumption \eqref{eq:entropy} is a higher integrability property for $f$, and thus it is {not} a bound on a macroscopic observable of the form $\int f(t,x,v)\varphi(v) \d v$.
The entropy density is a natural hydrodynamic quantity, but not an observable in the usual sense (linear in $f$).

Notice also that the entropy assumption (together with \eqref{eq:mass}) is significantly stronger than a control from below on the temperature $T$ in \eqref{eq:temp}. 
Indeed, the higher integrability assumption \eqref{eq:entropy} on the entropy density implies in particular that $f$ is absolutely continuous and cannot have too much mass on any set of small measure, while a bound from below on the temperature $T(t,x)$ only says that not all particles at $(t,x)$ have the same velocity, i.e., any $f(t,x,\cdot)$ different from a Dirac's delta has positive temperature.

This means that the assumptions in Imbert--Silvestre \cite{ImSi22} are still stronger than the ones proposed in Desvillettes--Villani \cite{DeVi05}.
This gives rise to the following open problem (explicitly mentioned in~\cite{ImSi22}): 
\[\begin{array}{c} \textit{Does the regularity program of Imbert--Silvestre remain valid if the entropy upper bound \eqref{eq:entropy}} \\ \textit{is replaced by weaker macroscopic bounds?} \end{array}\]

This is the question we study in this paper.

Our main results allow us to replace the upper bound on the entropy by a lower bound on the pressure
\[
\mathbb P(t,x) \geq p_0 {\rm Id_n}>0,
\]
or, equivalently,
\begin{align}
\label{eq:temperature}
\inf_{e \in \mathbb{S}^{n-1}} |e\cdot \mathbb P e| = \inf_{e \in \mathbb{S}^{n-1}} \int_{\R^n} f(t,x,v) |(v - \bar{v}) \cdot e|^2 \d v \ge p_0>0.
\end{align}
Notice that this condition allows for very singular distributions $f$ at any given $(t,x)$, and the only requirement is that we have ``\emph{positive temperature in all directions $e$}''.
In other words, the condition is only violated at $(t,x)$ when $f$ is concentrated on a hyperplane.

Actually, for our main results, it suffices to assume the weaker condition that at least two different eigenvalues of $(\mathbb P_{ij})_{ij}$ are positive, namely, that 
\begin{equation}
\label{eq:temperatureB}
\inf_{\sigma \in \mathbb{S}^{n-1}} \sup_{\substack{e\perp \sigma \\ e \in \mathbb{S}^{n-1}}}\int_{\R^n} f(t,x,v) |(v - \bar{v}) \cdot e|^2 \d v \ge p_0>0.
\end{equation}
This is equivalent to saying that we have ``positive temperature in at least two different directions'', i.e., that $f(t,x,\cdot)$ is not concentrated on a line.

In addition to this, we also need to assume that the $q$-th moment is finite for some $q >2$, i.e.,
\begin{align}
\label{eq:moment}
\int_{\R^n} f(t,x,v) |v|^q \d v \le M_q.
\end{align}

Note that the bounds on mass \eqref{eq:mass}, energy \eqref{eq:energy}, and entropy \eqref{eq:entropy},  imply \eqref{eq:moment} for \emph{all} $q > 2$; see  \cite[Theorem 1.3(ii)]{IMS20}.

Notice that both conditions \eqref{eq:temperatureB} and \eqref{eq:moment} are given in terms of macroscopic observables of the form $\int f(t,x,v)\varphi(v) \d v$.

Moreover, as explained below, we will show that replacing the lower bound on the pressure $\mathbb P(t,x)$ (equivalent to a  lower bound on ``directional temperatures'') by a lower bound on the temperature $T(t,x)$ would require completely new ideas.
Our hypotheses are, in some sense, the minimal ones under which the diffusion in Boltzmann's equation is still $n$-dimensional.

Our main result applies to the case of hard potentials $\gamma>0$, and reads as follows: 

\begin{theorem}
\label{cor:smoothness}
Let $s\in(0,1)$, $\gamma > 0$, $q > n$, and $\gamma + 2s\leq q$. Let $f$ be a solution to the Boltzmann equation in $(0,T) \times \R^n \times \R^n$ with $n \ge 2$ (see \autoref{def:solution}). Assume that $f$ satisfies \eqref{eq:mass},  \eqref{eq:temperatureB}, and \eqref{eq:moment} with $q > n$. 

Then, for any multi-index $k \in \mathbb{N}^{1+2n}$, and any $\tau > 0$ and $p \geq 0$, it holds
\begin{align*}
\big\|  |v|^p D^k f \big\|_{L^{\infty}([\tau,T] \times \R^n \times \R^n)} \le C_{k,p}, 
\end{align*}
where $C_{k,p}$ depends only on $n,s,\gamma,m_0,M_0, p_0,M_q,q,p,\tau,k$.
\end{theorem}

Notice that, in order to prove this result, the key point is to establish the case $k=0$, $p=0$, that is, an $L^\infty$ bound for $f$.
Indeed, once this case is established then the entropy bound \eqref{eq:entropy} automatically holds, and we can apply the results of Imbert--Silvestre  \cite{ImSi22}.

$L^\infty$ bounds were established in Silvestre \cite{Sil16} and Imbert--Mouhot--Silvestre \cite{IMS20} under the entropy assumption \eqref{eq:entropy}, together with \eqref{eq:mass} and \eqref{eq:energy}.
However, the proofs in \cite{Sil16} and \cite{IMS20} do \emph{not} work when one replaces the entropy bound by the pressure and moment bounds in these papers.

Our main contribution is to establish such $L^\infty$ bounds with a completely different method, allowing us to replace the entropy assumption by a lower bound on the pressure and some moment bounds.

\begin{theorem} \label{thm:entropy-finite}
Let $s\in(0,1)$, $\gamma \ge 0$, $q>n$, and $\gamma + 2s \leq q$. Let $f$ be a solution to the Boltzmann equation in $(0,T) \times \R^n \times \R^n$ with $n \ge 2$ (see \autoref{def:solution}). Assume that $f$ satisfies \eqref{eq:mass}, \eqref{eq:temperatureB}, and \eqref{eq:moment} with $q > n$. 

Then, for any $\tau > 0$ we have 
\[\big\| f\big\|_{L^{\infty}([\tau,T] \times \R^n \times \R^n)} \leq C,\]
 $C$ depending only on $n,s,m_0,M_0,p_0,M_q,q$, and $\tau$.

In particular, the entropy bound \eqref{eq:entropy} holds for some $H_0$ depending only on $n,s,m_0,M_0,p_0,M_q,q$, and $\tau$.
\end{theorem}

Notice also that the $L^\infty$ bound for positive times in \autoref{thm:entropy-finite} holds for $\gamma=0$ as well (Maxwellian molecules).
However, in order to deduce \autoref{cor:smoothness} we need to use the results in \cite{IMS20}, where in case $\gamma=0$ the decay for large velocities is inherited from the initial condition (while for $\gamma>0$ it is an inherent regularization for all positives times, independent of the initial condition).
This is the only reason why in \autoref{cor:smoothness} we need to assume $\gamma>0$.

Let us also mention the recent results in \cite{ImLo25}, where decay estimates for large velocities, independent of the initial condition, are proved for the Boltzmann equation with moderately soft potentials $\gamma+2s\geq0$, for any  $\gamma \in (-n,1]$. Notice, however, that the results in \cite{ImLo25} are for solutions in bounded domains, while our results are for $x$-periodic solutions. An extension of their results to $x$-periodic solutions would allow us to include also the case $\gamma = 0$ in \autoref{cor:smoothness}.

\subsection{Strategy of the proof}

The Boltzmann collision operator can be written via Carleman coordinates as follows
\begin{equation}
\label{eq:Qfg}
Q(f,g) = \mathcal{L}_{K_f} g + g(f \ast c_b |\cdot|^{\gamma}),
\end{equation}
where $c_b > 0$ is a constant, depending only on the Boltzmann collision kernel $B$, and $\mathcal{L}_{K_f}$ is an integro-differential operator of the form
\begin{align*}
\mathcal{L}_{K_f} g(v) = \int_{\R^n} (g(v+h) - g(v)) K_f(v,v + h) \d h.
\end{align*}

The kernel $K_f : \R^n \times \R^n \to [0,\infty]$ depends on the function $f$ as follows:
\begin{align}
\label{eq:Kf-def}
K_f(v,v') = \frac{2^{n-1}}{|v - v'|} \int_{w \perp v' - v} f(v + w) B(r , \cos \theta) r^{-n+2} \d w
\end{align}
with 
\begin{align}
\label{eq:cos-r-def}
r^2 = |v - v'|^2 + |w|^2, \qquad \cos \theta = \frac{w - (v - v')}{|w - (v - v')|} \cdot \frac{w - (v' - v)}{|w - (v' - v)|},
\end{align}
and satisfies the following pointwise upper and lower bound (see \cite[Corollary 4.2]{Sil16})
\begin{align}
\label{eq:K-bounds}
K_f(v,v+h) \asymp |h|^{-n-2s} \left(\int_{w \bot h} f(v+w) |w|^{\gamma + 2s + 1} \d w \right),
\end{align}
where the constants hidden behind the symbol $\asymp$ only depend on $B$, and will be neglected in the sequel.

Thus, the Boltzmann equation can be written as a nonlinear kinetic integro-differential equation, where the kernel $K_f$ depends on the solution $f$ itself.

\subsubsection{Ellipticity conditions}  A key observation in the program of Imbert--Silvestre is that, if we have a priori bounds on the mass, energy, and entropy densities \eqref{eq:mass}-\eqref{eq:energy}-\eqref{eq:entropy}, then the kernel $K_f$ is uniformly elliptic in the following sense:
\begin{equation}\label{cone}
K_f(v,v+h) \geq \frac{\lambda}{|h|^{n+2s}}\,\1_{\mathcal{C}_v}(h)\quad \textrm{for some cone $\mathcal{C}_v$,}
\end{equation}
where $\lambda>0$ and the cone $\mathcal{C}_v$ depend only on $m_0,M_0,E_0,H_0$, and $v$.

The existence of these cones $\mathcal{C}_v$ comes from the fact that we have a uniform bound on the entropy density.
Unfortunately, if we only assume a lower bound on the pressure \eqref{eq:temperature}, then all the mass of $f$ could be concentrated on a set of zero measure, and \eqref{cone} could fail.

Still, we prove that under our macroscopic assumptions (bounds on mass, pressure, and moments) we have the following weaker ellipticity conditions for $K_f$.

\begin{proposition}
\label{prop:properties_kinetic_kernels}
Let $s\in (0, 1)$, and let $f$ be nonnegative and satisfying \eqref{eq:mass}, \eqref{eq:temperatureB}, and \eqref{eq:moment} for some $ q > 2$. Then, the Boltzmann kernel $K = \tilde K_f$ given by \eqref{eq:Ktildedef} with $v_0\in \R^n$ and $\gamma+2s\in [0, q]$ satisfies:
 \begin{itemize}
\item[(i)] (Upper bound) For any $r > 0$ and any $v \in B_2$:
\begin{align*}
\int_{\R^n \setminus B_r} K(v,v+h) \d h + \int_{\R^n \setminus B_r} K(v+h,v) \d h \le \Lambda r^{-2s}.
\end{align*}
\item[(ii)]  (Nondegeneracy) For any $r > 0$ and $v \in B_2$
\begin{align}
\label{eq:nondeg-intro} \inf_{e \in \mathbb{S}^{n-1}} \int_{B_r} K(v,v+h) (h \cdot e)_+^2 \d h \ge \lambda r^{2-2s}> 0.
\end{align}
\item[(iii)]  (Coercivity) For any $g$ supported in $B_2$:
\begin{align}
 \label{eq:coercivity-intro} \int_{B_2} \int_{\R^n} (g(v') - g(v))^2 K(v,v') \d v \d v' \ge \lambda [g]_{H^s(\R^n)}^2 - \Lambda \Vert g \Vert_{L^2(\R^n)}^2,
\end{align}
\item[(iv)] (Cancellation condition) For any $r \in (0,1)$ and $v \in B_2$:
\begin{align*}
\left| \int_{B_r} (K(v,v+h) - K(v+h,v)) \d h \right| &\le \Lambda r^{-2s}, \\
\left| \int_{B_r} (K(v,v+h) - K(v+h,v)) h \d h  \right| &\le \Lambda (1 + r^{1-2s}) ~~ \text{ if } s \ge \frac{1}{2},
\end{align*}
\end{itemize}
uniformly in $v_0$, for some constants $\lambda$ and $\Lambda$ depending only on $n$, $s$, $\gamma$, $m_0$, $M_0$, $p_0$, $M_q$, and $q$. 
\end{proposition}

The nondegeneracy condition (ii) is the minimal hypothesis to ensure that the diffusion given by $K_f$ is really $n$-dimensional; see \cite[Proposition~2.2.1]{FeRo24}.

The upper bounds (i) and (iv) are rather simple to prove, since they do not rely on any entropy or pressure lower bound, and have already been established in \cite{ImSi22}. 
In contrast, the verification of the nondegeneracy (ii) and coercivity (iii) are more delicate, and were established in \cite{ImSi20b,ImSi22} under the assumption \eqref{cone} (see also \cite{ChSi20}).

A key contribution of the current paper consists in the verification of the two conditions (ii) and (iii) under pressure and moment bounds. 
We will establish these properties in \autoref{thm:change-of-variables-nondegeneracy} and \autoref{thm:change-of-variables-coercivity} respectively.
To verify the condition (iii), we rely on the results of Gressmann--Strain \cite{GrSt11}.

It is important to notice that the conditions (ii) and (iii) can \emph{fail} if we only assume the energy bound \eqref{eq:energy} instead of \eqref{eq:moment} for some $q>2$; see \autoref{rmk:qmoments}.
This is a first reason why we need to assume higher order moments. 

\subsubsection{From ellipticity to regularity} \autoref{prop:properties_kinetic_kernels} is a crucial ingredient for our proof, as it tells us that under our macroscopic assumptions, the diffusion coming from $K_f$ is $n$-dimensional, and thus there is hope to establish some regularity results.

In the program of Imbert--Silvestre, some of the main steps of the proof are the following:
\begin{itemize}
\item Prove an $L^\infty$ bound for solutions, subject to the macroscopic bounds on mass, energy, entropy. This was done by Imbert--Mouhot--Silvestre \cite{Sil16,IMS20}.

\item Establish a $C^\alpha-L^\infty$ estimate, and deduce that solutions are $C^\alpha$ \cite{ImSi20b,ImSi22}.

\item Establish a higher order Schauder estimate, and deduce that solutions are $C^\infty$ \cite{ImSi21,ImSi22}.
\end{itemize}

The entropy bound \eqref{eq:entropy} is crucially needed for the $L^\infty$ bound.
Indeed, the proof of \cite{IMS20} does not work if we only assume (i)-(ii)-(iii)-(iv) above, and thus we need a completely different proof under these weaker assumptions.
Notice also that this is the only missing step, because once we have an $L^\infty$ bound for solutions then the entropy is automatically bounded and we can apply the existing results.

Our proof of the $L^\infty$ bound (\autoref{thm:entropy-finite}) relies on the  $C^{\alpha}-L^{\infty}$ estimate from \cite{ImSi20b}, which holds exactly under the assumptions (i)-(ii)-(iii)-(iv).
Namely, the idea is that, under our macroscopic bounds \eqref{eq:mass}, \eqref{eq:temperature} (or \eqref{eq:temperatureB}) and \eqref{eq:moment}, any solution will satisfy a bound of the type
\[\|f\|_{C^\alpha} \leq C \|f\|_{L^\infty},\]
for some $C$ that does not depend on $\|f\|_{L^\infty}$ (in particular we do not need the bound on the entropy here).
If this was true globally (which is not the case), by an interpolation argument (in kinetic spaces) we would show 
\[\|f\|_{C^\alpha} \lesssim \|f\|_{L^\infty} \leq C_\delta \|f\|_{L^1} +\delta \|f\|_{C^\alpha},\]
and then we could reabsorb the term on the RHS to deduce $\|f\|_{C^\alpha} \lesssim \|f\|_{L^1}$, and in particular
\[\|f\|_{L^\infty} \lesssim \|f\|_{L^1},\]
which is the estimate we want.
This type of argument works well for harmonic functions (or elliptic equations), but it is much more delicate here because of the kinetic scaling, the degeneracy of the kernel $K_f$ as $v\to \infty$, and the exponent $\gamma$ in the equation.
Despite all this, we manage to make the argument work provided that we have finite moments of some order $q>n$.

\subsubsection{Related results}
Let us close this subsection by emphasizing that our technique to prove the $C^{\alpha}-L^1$ estimate (resp. \autoref{thm:entropy-finite}) would also work for linear kinetic equations of the form
\begin{align}
\label{eq:linear-intro}
\partial_t f + v \cdot \nabla_x f = \mathcal{L}_K f + h,
\end{align}
where $\mathcal{L}_K$ is a nonlocal operator with kernel $K$ satisfying (i)-(ii)-(iii)-(iv) from \autoref{prop:properties_kinetic_kernels}. Our proof heavily relies on the H\"older estimate which was developed in \cite{ImSi20b} for \textit{linear} nonlocal kinetic equations with bounded measurable coefficients and then applied to the \textit{nonlinear} Boltzmann equation.
Recently, a great deal of attention has been paid to the study of regularity properties for linear nonlocal kinetic equations like \eqref{eq:linear-intro}. A closely related question to the De Giorgi--Nash--Moser type results from \cite{ImSi20b} is the validity of a Harnack inequality for nonlocal kinetic equations. Interestingly, it turns out that the Harnack inequality fails to hold for solutions to \eqref{eq:linear-intro}, already in case $\mathcal{L}_K = (-\Delta_v)^s$ is the fractional Laplacian (see \cite{KaWe24}). Let us refer to \cite{Sto19,Loh23,Loh24b,APP24} for further results on pointwise regularity estimates for solutions to \eqref{eq:linear-intro}.
Moreover, let us mention \cite{ChZh18, NiZa21, NiZa22, Nie22} for results on nonlocal kinetic $L^p$ maximal regularity and  \cite{ImSi21, HWZ20, Loh24} where Schauder-type regularity estimates have been established.

\subsection{Convergence to equilibrium}

Exactly as in \cite{ImSi22}, an immediate consequence of our \autoref{cor:smoothness} is the following improvement of the main result in \cite{DeVi05}:

\begin{corollary}
\label{cor:convergence}
Let $s\in(0,1)$, $\gamma > 0$, $q > n$, and $\gamma + 2s\leq q$. Let $f$ be a solution to the Boltzmann equation in $(0,\infty) \times \R^n \times \R^n$ with $n \ge 2$ (see \autoref{def:solution}). 
Assume that $f$ satisfies globally \eqref{eq:mass},  \eqref{eq:temperatureB}, and \eqref{eq:moment} with $q > n$. 

Then, $f$ converges to a Maxwellian as $t\to\infty$ as described in \cite[Theorem 2]{DeVi05}.
\end{corollary}

In other words, if the macroscopic observables in \eqref{eq:mass},  \eqref{eq:temperatureB}, and \eqref{eq:moment} remain controlled, then $f$ will converge to equilibrium as $t\to\infty$ faster than any algebraic rate $O(t^{-k})$.

\subsection{The grazing collision limit}

The Boltzmann equation converges formally to the Landau equation as $s\to1$ (see, e.g., \cite{Vi02}), provided that the collision kernel has the appropriate normalizing factor
\begin{equation} \label{s-to-1-kernel}
b(\cos \theta) \asymp (1-s)|\sin (\theta/2)|^{-(n-1)-2s}.
\end{equation}

An open problem after the results of \cite{ImSi22} is to establish regularity estimates (like those in \autoref{cor:smoothness}) that remain uniform as $s\to 1$.
As explained in \cite{ImSi22}, the main difficulty lies in the $L^\infty$ estimates and decay for large velocities from \cite{IMS20}.
The proof in \cite{IMS20} heavily uses the nonlocality of the equation, and thus their estimates cannot be made uniform as $s\to 1$.

Another advantage of the method we introduce in this paper is that the new $L^\infty$ estimate we establish here (\autoref{thm:entropy-finite}) can be made uniform in the grazing collision limit $s\to 1$.

\begin{theorem} \label{thm:grazing}
Let $s_0 \in (0,1)$, $s\in [s_0,1)$, $\gamma \ge 0$, $q>n$, and $\gamma + 2s \leq q$. 
Let $f$ be a solution to the Boltzmann equation as in \autoref{def:solution}, with the normalization factor \eqref{s-to-1-kernel}. Assume that $f$ satisfies \eqref{eq:mass}, \eqref{eq:temperatureB}, and \eqref{eq:moment}. Then, for any $\tau > 0$ we have 
\[\big\| f\big\|_{L^{\infty}([\tau,T] \times \R^n \times \R^n)} \leq C_0,\]
  with $C_0$ depending only on $n,s_0,m_0,M_0,p_0,M_q,q,$ and $\tau$.

Moreover, if $f$ satisfies \eqref{eq:moment} for all $q > n$ and if $\gamma \ge 0$, then, for any $\tau > 0$ and $p \ge 0$ we have 
\[\big\||v|^p f\big\|_{L^{\infty}([\tau,T] \times \R^n \times \R^n)} \leq C_p \]
with $C_p$ depending only on $n,s_0,\gamma,m_0,M_0,p_0,p,\tau$, and on $M_{p+n+1}$.
\end{theorem}

Note that once we have the $L^\infty$ estimate from \autoref{thm:entropy-finite}, we still use the results in \cite{IMS20} to deduce fast decay for large velocities, and hence the estimate in \autoref{cor:smoothness} is still not uniform as $s\to 1$.
The use of \cite{IMS20} can be avoided entirely, for example, when assuming that we have finite moments of \emph{all} orders $q>1$ (instead of \emph{some} $q>n$). In that case, \autoref{thm:grazing} becomes a robust analog of the main result in \cite{IMS20}. Then, the only missing ingredient to obtain a uniform version of \autoref{cor:smoothness} is to prove robust Schauder estimates for nonlocal kinetic equations (see \cite{ImSi21}).

\subsection{The Landau equation} 
Quite interestingly, \autoref{thm:grazing} seems to be new even for the Landau equation, which corresponds to the limit $s=1$, and is given by
\begin{align}
\label{eq:Landau} \partial_t f + v \cdot \nabla_x f = \nabla_v \cdot [A \nabla_v f] + b \cdot \nabla_v f + c f,
\end{align}
where
\begin{align}
\label{eq:a_Landau} A(t,x,v) &= a_{n,\gamma} \int_{\R^n} \left(I - \frac{w}{|w|} \otimes \frac{w}{|w|} \right) |w|^{\gamma + 2} f(t,x,v - w) \d w,\\
\label{eq:b_Landau} b(t,x,v) &= b_{n,\gamma} \int_{\R^n} w |w|^{\gamma} f(t,x,v-w) \d w,\\
\label{eq:c_Landau} c(t,x,v) &= c_{n,\gamma} \int_{\R^n} |w|^{\gamma} f(t,x,v-w) \d w.
\end{align}
Here, $a_{n,\gamma},b_{n,\gamma},c_{n,\gamma}$ are constants with $a_{n, \gamma} > 0$ (which we will neglect in the sequel), and $\gamma > -n$. 

\autoref{thm:grazing} was known for the Landau equation under the entropy bound assumption~\eqref{eq:entropy} in the cases $\gamma \le 0$ (see \cite{HeSn20}) and $\gamma \in (0,1]$ (see \cite{Sne20}), and for space-homogeneous solutions (see \cite{DeVi00}).

In particular, we deduce the following smoothness result for the Landau equation with hard potentials.

\begin{corollary}
\label{cor:Landau}
Let $q > n$. Let $f$ be a weak solution to the Landau equation with $\gamma \ge 0$ and $\gamma + 2 \le q$. Assume that $f$ satisfies \eqref{eq:mass}, \eqref{eq:temperatureB}, and \eqref{eq:moment}. Then, for any $\tau > 0$ we have 
\[\big\| f\big\|_{L^{\infty}([\tau,T] \times \R^n \times \R^n)} \leq C_0,\]
  with $C_0$ depending only on $n,m_0,M_0,p_0,M_q,q,$ and $\tau$.

Moreover, if $f$ satisfies \eqref{eq:moment} for all $q>n$ and if $\gamma \ge 0$, then, for any multi-index $k \in \mathbb{N}^{1+2n}$, and any $\tau > 0$ and $p \geq 0$ it holds
\begin{align*}
\big\| |v|^p D^k f \big\|_{L^{\infty}([\tau,T] \times \R^n \times \R^n)} \le C_{k, p}, 
\end{align*}
where $C_{k, p}$ depends only on $n,\gamma,m_0,M_0, p_0,p, \tau,k,$ and on all $M_q$ for $q > 0$.
\end{corollary}

Note that, in fact, for given $k,p$, there exists $q_0$ so that $C_{k,p}$ only depends on $M_{q_0}$.

\subsection{Acknowledgments}
X. F. was supported by the Swiss National Science Foundation (SNF grant  PZ00P2\_208930), by the Swiss State Secretariat for Education, Research and Innovation (SERI) under contract number MB22.00034, and by the AEI project PID2021-125021NA-I00 (Spain).
 
X. R. and M. W. were supported by the European Research Council under the
Grant Agreement No. 101123223 (SSNSD), and by the AEI project PID2021-125021NA-I00 (Spain). Furthermore, X. R. was supported by the grant RED2022-134784-T funded by  AEI/10.13039/501100011033, by AGAUR Grant 2021 SGR 00087 (Catalunya), and by the Spanish State Research Agency through the Mar\'ia de Maeztu Program for Centers and Units of
Excellence in R{\&}D (CEX2020-001084-M).

\subsection{Outline}

This article is structured as follows. In Section \ref{sec:prelim} we introduce (recall) the change of variables and kinetic H\"older spaces. In Section~\ref{sec:aux}, we present several consequences of the pressure and moment bounds. In Section \ref{sec:nondeg} and Section \ref{sec:coercive} we prove that the Boltzmann kernels are nondegenerate and coercive in the sense of (ii) and (iii) above, respectively. In Section \ref{sec:main} we give the proofs of our main results for the Boltzmann equation. Finally, in Section \ref{sec:Landau} we explain how to adapt our technique to the Landau equation.

\section{Preliminaries}
\label{sec:prelim}

In this section we introduce the change of variables from \cite{ImSi22} that preserves the geometry of the Boltzmann equation and is crucial in order to deduce global H\"older estimates, as well as some definitions on kinetic H\"older spaces.

We start with the definition of the (kinetic) cylinder adapted to $(t, x, v)$ variables and the kernel's singularity (in fact, the singularity of its angular part): 

Given a point $z_0 = (t_0, x_0, v_0)\in \R^{1+2n}$, we denote by $Q_r(z_0)$ the kinetic cylinder of radius $r$ and centered at $z_0$,
\begin{equation}
\label{eq:Qr}
Q_r(z_0) := \{(t, x, v)\in \R^{1+2n} : t_0 - r^{2s} < t \le t_0, |x -x_0 - (t - t_0) v_0|< r^{1+2s}, |v  - v_0| < r\}. 
\end{equation}
We will denote by $Q_r = Q_r(0, 0, 0)$. 

\subsection{Change of variables}

Note that by verifying nondegeneracy \eqref{eq:nondeg-intro} and coercivity \eqref{eq:coercivity-intro} for $K_f$ we can obtain a H\"older estimate in $B_1$ by application of the main result in \cite{ImSi20b} (see also \autoref{prop:Holder-estimate-kinetic}). In order to obtain a global H\"older regularity estimate in $B_1(v_0)$ for some $v_0 \in \R^n$, we need to verify \eqref{eq:nondeg-intro} for any $v \in B_2(v_0)$ and \eqref{eq:coercivity-intro} for functions $g$ supported in $B_2(v_0)$. It turns out that the verification of these translated versions of \eqref{eq:nondeg-intro} and \eqref{eq:coercivity-intro} is not for free  since the ellipticity constants degenerate (or explode) as $|v| \to \infty$. Clearly, in order to obtain global regularity estimates for solutions to the Boltzmann equation, it is crucial to have uniform ellipticity for all velocities.

In \cite{ImSi22}, this problem is solved by introducing a suitable change of variables that preserves the geometry of the Boltzmann equation and under which the ellipticity constants remain controlled:

Let $\gamma + 2s \ge 0$. Given $t_0 \in \R$, $x_0 \in \R^n$, and $v_0 \in \R^n$, we consider 
\begin{align*}
\tilde{f}(t,x,v) = f(\tilde{t},\tilde{x},\tilde{v}),
\end{align*}
where
\begin{align*}
(\tilde{t},\tilde{x},\tilde{v}) = \mathcal{T}_0(t,x,v) = \left\{\begin{array}{ll}
\left( t_0 + \frac{t}{|v_0|^{\gamma + 2s}} , x_0 + \frac{\tau_0 (x) + t v_0}{|v_0|^{\gamma + 2s}} , v_0 + \tau_0 (v) \right),& \quad\text{if}\quad |v_0|\ge 2,\\[0.3cm]
\left(t_0+t, x_0+x+tv_0, v_0+v\right),& \quad\text{if}\quad |v_0|< 2,
\end{array}
\right.
\end{align*}
and $\tau_0 : \R^n \to \R^n$ is defined as the following transformation
\begin{align*}
\tau_0(a v_0 + w) = \left\{\begin{array}{ll} \frac{a}{|v_0|} v_0 + w ~~ \forall w \perp v_0, ~~ a \in \R,& \quad\text{if}\quad |v_0|\ge 2,\\[0.3cm]
a v_0 + w,& \quad\text{if}\quad |v_0|<2.
\end{array}
\right.
\end{align*} 

Let us also introduce the following sets, where we recall that $Q_r$ is the  kinetic cylinder in $\R^{1+2n}$ (given by \eqref{eq:Qr}) and $B_r$ is the usual ball in $\R^n$:
\begin{align*}
\cE_r(z_0) = \mathcal{T}_0(Q_r), \qquad E_r(v_0) = v_0 + \tau_0(B_r).
\end{align*}
Clearly, when $|v_0| < 2$, it holds $Q_r(z_0) = \mathcal{E}_r(z_0)$.
Moreover, note that when $f$ solves the Boltzmann equation in $\cE_1(z_0)$, then $\tilde{f}$ solves
\begin{align*}
\partial_t \tilde{f} + v \cdot \nabla_x \tilde{f} = \mathcal{L}_{\tilde{K}_f} \tilde{f} + \tilde{g} ~~ \text{ in } Q_1,
\end{align*}
where
\begin{align}
\label{eq:Ktildedef} \tilde{K}_f(t,x,v,v+h) &= 
\left\{\begin{array}{ll}
|v_0|^{-1-\gamma-2s}K_f(\tilde{t},\tilde{x},\tilde{v},v_0 + \tau_0(v+h)),& \quad\text{if}\quad |v_0|\ge 2,\\[0.3cm]
 K_f(\tilde{t},\tilde{x},\tilde{v},v_0 + v+h),& \quad\text{if}\quad |v_0|< 2,
\end{array}
\right.
\\[0.3cm]
\tilde{g}(t,x,v) &= \left\{\begin{array}{ll}
c_b |v_0|^{-\gamma-2s}f(\tilde{t},\tilde{x},\tilde{v})(f \ast |\cdot|^{\gamma})(\tilde{t},\tilde{x},\tilde{v}),& \quad\text{if}\quad |v_0|\ge 2,\\[0.3cm]
c_b f(\tilde{t},\tilde{x},\tilde{v})(f \ast |\cdot|^{\gamma})(\tilde{t},\tilde{x},\tilde{v}),& \quad\text{if}\quad |v_0|< 2.
\end{array}
\right.
\end{align}
Note that $\tilde{K}_f$ is still homogeneous and satisfies the non-divergence form symmetry condition.

In order to obtain a global H\"older estimate for solutions to the Boltzmann equation, we need to verify the nondegeneracy \eqref{eq:nondeg-intro} and coercivity \eqref{eq:coercivity-intro} with $K = \tilde{K}_f$ for any $v_0 \in \R^n$.

\subsection{Kinetic H\"older spaces}

We  also recall the notion of kinetic distance and the corresponding H\"older spaces: 
\begin{definition}
\label{def:kin_dist}
Given two points $z_i = (t_i,x_i,v_i) \in \R^{1+2n}$, $i = 1,2$, we define the kinetic distance
\begin{align*}
d_\ell(z_1,z_2) = \min_{w \in \R^n} \left\{ \max \left( |t_1 -t_2|^{\frac{1}{2s}} , |x_1 - x_2  - (t_1-t_2)w|^{\frac{1}{1+2s}} , |v_1 - w|, |v_2 - w| \right) \right\}.
\end{align*}
Given a set $D \subset \R^{1+2n}$ and $\alpha \in [0,1)$, we say that a function $f : D \to \R$ is $\alpha$-H\"older continuous at $z_0 \in \R^{1+2n}$ if 
\begin{align*}
|f(z) - f(z_0)| \le C_{z_0} d_\ell(z,z_0)^{\alpha}\quad\text{for all}\quad z\in D. 
\end{align*}
We define the set $C_\ell^{\alpha}(D)$ as the set of all functions $f : D \to \R$ that are $\alpha$-H\"older continuous at any $z_0 \in D$ with a constant $C_{z_0}$ that is independent of $z_0$, and we define $[f]_{C^{\alpha}_\ell(D)}$ as the supremum over all $C_{z_0}$, $z_0 \in D$. Moreover, we set
\begin{align*}
\Vert f \Vert_{C^{\alpha}_\ell(D)} = \Vert f \Vert_{L^{\infty}(D)} + [f]_{C^{\alpha}_\ell(D)}, \qquad [f]_{C^0_\ell(D)} = \Vert f \Vert_{L^{\infty}(D)}.
\end{align*}
Moreover, given $p > 0$, we define 
\begin{align*}
\Vert f \Vert_{C^{\alpha}_{\ell,p}((\tau,T) \times \R^n \times \R^n)} &= \sup \left\{ (1 + |v|)^p \Vert f \Vert_{C_\ell^{\alpha}(Q_r(z))} : r \in (0,1], ~ Q_r(z) \subset (\tau,T) \times \R^n \times \R^n \right\},\\
\Vert f \Vert_{L^{\infty}_{t,x}L^1_{\ell,p}((\tau,T) \times \R^n \times \R^n)} &= \sup \left\{ \Vert f \Vert_{L^{\infty}_{t,x}L^1_v(Q_r(z) ; (1 + |v|)^p \d v)} : r \in (0,1], ~ Q_r(z) \subset (\tau,T) \times \R^n \times \R^n \right\}.
\end{align*}
We say that $f \in C^{\alpha}_{\ell,\text{fast}}$ if for any $p > 0$ and all $r \in (0,1]$ there is $C_p > 0$ such that for all $Q_r(z) \subset (\tau,T) \times \R^n \times \R^n)$ it holds $\Vert f \Vert_{C^{\alpha}_{\ell,p}(Q_r(z))} \le C_p$.
\end{definition}

\section{Auxiliary lemmas}
\label{sec:aux}

In this section, we present an important consequence of the pressure and moment bounds (for $q > 2$), which will turn out to be useful in the proofs of the nondegeneracy \eqref{eq:nondeg-intro} and coercivity \eqref{eq:coercivity-intro} of the Boltzmann kernels.

It is well-known that by \eqref{eq:mass} and \eqref{eq:energy}, solutions to the Boltzmann equation have some positive mass which can be located in a ball around the center of mass $\bar{v}$. Moreover, due to the pressure lower bound and the moment bound for some $q > 2$, the location of this mass can be further specified. In particular, the following result, which is the main result of this section, states that solutions have some positive mass located outside any linear tube of radius $\delta$.

\begin{proposition}
\label{prop:subtract-tube}
Assume that $f$ is nonnegative and satisfies \eqref{eq:mass}, \eqref{eq:temperatureB}, and \eqref{eq:moment} for some $q > 2$. There exist $R \ge 2$,  and $\delta, c > 0$, depending only on $m_0,M_0,p_0, M_q$, and $q$, such that for any line $L\subset \R^n$, denoting $L_\delta := \{x : {\rm dist}(x, L) < \delta\}$ the linear tube of radius $\delta$ around $L$, we have 
\begin{align*}
\int_{B_R  \setminus L_{\delta}} f(w) \d w \ge c.
\end{align*}
\end{proposition}

\begin{remark}
\label{rmk:qmoments}
Note that \autoref{prop:subtract-tube} does not hold without the assumption that the $q$th moment is finite for some $q > 2$, \eqref{eq:moment}. Indeed, when only boundedness of the mass \eqref{eq:mass}, energy \eqref{eq:energy}, and pressure \eqref{eq:temperature} are assumed, then one can construct counterexamples to \autoref{prop:subtract-tube}. For simplicity, we only give a counterexample in 2D. However, a similar construction also works in higher dimensions. Consider for $R > 1$ the sets 
\begin{align*}
A_1 = (-R^{-3},R^{-3}) \times (-R,R), \quad A_2 = (-R,R) \times (-R^{-3}, R^{-3}), \quad A = (-R^{-1},R^{-1}) \times (-R^{-1},R^{-1}),
\end{align*}
and define $f_R(v) = \1_{A_1 \cup A_2}(v) + R^2 \1_{A}(v)$. Then by construction, \eqref{eq:mass} holds true with $m_0 := 4$, $M_0 := 8 \ge 4 + 4 R^{-2}$. Moreover, it holds \eqref{eq:energy} with $E_0 := C_1$ and \eqref{eq:temperature} with $c_1E_0/M_0 \ge c_1 C_1 /8 =: p_0$ for some $0 < c_1 < C_1 < \infty$. In particular, $m_0,M_0,E_0$, and $p_0$, can be chosen independent of $R$. Moreover, note that for any $q > 2$, it holds for the $q$th moment of $f_R$ that $v^{(q)} \ge C_2 R^{q-2}$ for some $C_2 > 0$, so \eqref{eq:moment} fails for any $q > 2$, by taking $R \to \infty$. Moreover, note that $f_R$ violates the property in \autoref{prop:subtract-tube}. Indeed, given any $\delta \in (0,1)$, we let $L = \R e_2$, and observe that for $R^{-1} < \delta$ it holds
\begin{align*}
\int_{\R^n \setminus L_{\delta}} f_R(w) \d w = |A_2 \setminus L_{\delta}| = \left| \big( (-R,R) \times (-R^{-3} , R^{-3}) \big) \setminus \big (-\delta,\delta) \times \R^n \big) \right| \le 4R^{-2} \to 0.
\end{align*}
Since the right-hand side vanishes as $R \to \infty$, the statement of \autoref{prop:subtract-tube} fails for $f_R$.
\end{remark}

The proof of \autoref{prop:subtract-tube} requires some preparatory work. We start with the following lemma.

\begin{lemma}
\label{lem:1}
Assume that $f$ is nonnegative and satisfies \eqref{eq:mass} and \eqref{eq:temperatureB}. Then, for any $0 < \lambda < p_0$ there exists $\eta >0$ depending only on $M_0$, $p_0$, and $\lambda$, such that  
\begin{align*}
 \sup_{\substack{e\perp \sigma \\ e \in \mathbb{S}^{n-1}}} \int_{\{ |w \cdot e| \ge \eta \}} f(\bar{v} + w) |w \cdot e|^2 \d w \ge \lambda\qquad\text{for any}\quad \sigma \in \mathbb{S}^{n-1}.
\end{align*} 
\end{lemma}

\begin{proof}
We have for $\eta >0$:
\begin{align*}
\int_{ \{ |w \cdot e| \le \eta \}} f(\bar{v} + w) |w \cdot e|^2 \d w \le \eta^2 \int_{\R^n} f(\bar{v} + w) \d w \le \eta^2 M_0,
\end{align*}
so that 
\begin{align*}
\int_{\{ |w \cdot e| \ge \eta \}} f(\bar{v} + w) |w \cdot e|^2 \d w &= \int_{\R^n} f(\bar{v} + w) |w \cdot e|^2 \d w   - \int_{ \{ |w \cdot e| \le \eta \} } f(\bar{v} + w) |w \cdot e|^2 \d w\\
&\ge   \int_{\R^n} f(\bar{v} + w) |w \cdot e|^2 \d w  - \eta^2 M_0.
\end{align*}
Taking the supremum for $e\in \mathbb{S}^{n-1}$ with $e\perp \sigma$, and thanks to \eqref{eq:temperatureB}, we get  the desired result by taking  $\eta =   \left(\frac{p_0-\lambda}{M_0}\right)^{1/2}$.
\end{proof}

Moreover, if ($2+\eps$)-moments are finite we can ensure that the mass from \autoref{lem:1} is contained in a large ball.

\begin{lemma}
\label{lemma:mass-location}
Assume that $f$ is nonnegative and satisfies \eqref{eq:mass}, \eqref{eq:temperatureB}, and \eqref{eq:moment} for some $q > 2$. Then, there exist $R,\lambda > 0$ depending only on $m_0,M_0, p_0,M_q$, and $q$,  and $\delta_0 \in (0,1)$ depending only on $M_0$ and $p_0$ such that 
\begin{align*}
 \int_{B_R \cap \{{\rm dist}(\cdot, \langle \sigma \rangle ) \ge \delta_0\}} f(\bar v +w) \d w \ge \lambda\qquad\text{for any}\quad \sigma \in \mathbb{S}^{n-1},
\end{align*} 
where $\langle \sigma\rangle:= \{ t\sigma : t\in \R\}$ denotes the line spanned by $\sigma\in \mathbb{S}^{n-1}$. 
\end{lemma}

\begin{proof}
First, note that by \autoref{lem:1}, there exists $\delta_0 > 0$ such that we have
\begin{align}
\label{eq:mass-location-help-1}
 \sup_{\substack{e\perp \sigma \\ e \in \mathbb{S}^{n-1}}}\int_{\{ |w \cdot e| \ge \delta_0 \}} f(\bar v + w) |w \cdot e|^2 \d w \ge \frac{p_0}{4}.
\end{align}
Moreover, by \eqref{eq:moment}, we have that for any $\rho > 0$ and using that $(a+b)^q\le C_q(a^q+b^q)$ for all $a, b> 0$:
\begin{align*}
\int_{(\R^n \setminus B_\rho) \cap \{ |w\cdot e| \ge\delta_0 \}} f(\bar v +w) |w \cdot e|^2  \d w &\le \rho^{2-q} \int_{\R^n \setminus B_\rho(\bar{v})} f(w)  |w-\bar{v}|^q \d w\\
&\le C_q \rho^{2-q} (M_q + (M_1/m_0)^q M_0) =: C \rho^{2-q} ,
\end{align*}
where we used that $|\bar{v}| \le M_1/m_0$. Let us now choose $\rho = R$ so large that $R^{2-q}C \le p_0/8$. Then, using \eqref{eq:mass-location-help-1} and taking a supremum, we deduce
\begin{align*}
 \sup_{\substack{e\perp \sigma \\ e \in \mathbb{S}^{n-1}}}\int_{(\R^n \setminus B_R ) \cap \{ |w \cdot e| \ge \delta_0 \} }  f(\bar v + w)|w \cdot e|^2 \d w &\le C R^{2-q}  \le \frac{p_0}{8} \\
& \le \frac{1}{2} \sup_{\substack{e\perp \sigma \\ e \in \mathbb{S}^{n-1}}}\int_{\{ |w\cdot e| \ge \delta_0 \} } f(\bar v + w) |w \cdot e|^2 \d w.
\end{align*}
This implies (by the subadditivity of the supremum)
\begin{align*}
 \sup_{\substack{e\perp \sigma \\ e \in \mathbb{S}^{n-1}}} \int_{(\R^n \setminus B_R) \cap \{ |w \cdot e| \ge \delta_0 \} }  f(\bar v +w) |w \cdot e|^2\d w \le  \sup_{\substack{e\perp \sigma \\ e \in \mathbb{S}^{n-1}}}\int_{B_R \cap \{ |w \cdot e| \ge \delta_0 \} } f(\bar v +w)|w \cdot e|^2  \d w,
\end{align*}
and therefore, again by subadditivity,
\begin{align*}
\frac{p_0}{4}\le  \sup_{\substack{e\perp \sigma \\ e \in \mathbb{S}^{n-1}}}\int_{\{|w\cdot e | \ge \delta_0 \} }  f(\bar v + w) |w \cdot e|^2 \d w \le 2  \sup_{\substack{e\perp \sigma \\ e \in \mathbb{S}^{n-1}}}\int_{B_R  \cap \{|w\cdot e| \ge \delta_0 \} } f(\bar v + w)|w \cdot e|^2  \d w.
\end{align*}
This implies
\begin{align*}
\frac{p_0}{8R^2} \le  \sup_{\substack{e\perp \sigma \\ e \in \mathbb{S}^{n-1}}} \int_{B_R  \cap \{|w\cdot e| \ge \delta_0 \} } f(\bar v + w) \d w,
\end{align*}
which directly gives the desired result, since $\{|w\cdot e| \ge \delta_0 \}\subset \{{\rm dist}(\cdot, \langle\sigma\rangle)\ge \delta_0\}$ for any $e\perp \sigma$ with $e\in \mathbb{S}^{n-1}$. 

\end{proof}

The following lemma implies that the mass is not concentrated on one side of the center of mass $\bar{v}$:

\begin{lemma}
\label{lemma:center-of-mass}
Assume that $f$ is nonnegative and satisfies \eqref{eq:mass} and \eqref{eq:energy}. Moreover, assume that there exist $\eta, \lambda_1 > 0$ such that for some $e_0 \in \mathbb{S}^{n-1}$ it holds:
\begin{align}
\label{eq:center-of-mass-ass}
\int_{\{ w\cdot e_0 \ge \eta \}} f(\bar v + w) \d w \ge \lambda_1.
\end{align}
Then, there exist $\varrho, \lambda_2 > 0$, depending only on $\lambda_1,\eta,m_0,M_1,E_0$, such that
\begin{align*}
\int_{B_R  \cap \{ w\cdot e_0 \le 0 \}} f(\bar v +w) \d w \ge \lambda_2.
\end{align*}
\end{lemma}

\begin{proof}
First, note that for any $\varrho > 0$, we have
\begin{align}
\label{eq:center-of-mass-help-1}
\int_{\R^n \setminus B_\varrho(\bar{v})} f(w) |(w - \bar{v}) \cdot e_0| \d w \le \varrho^{-1} \int_{\R^n \setminus B_\varrho(\bar{v})} f(w) |w - \bar{v}|^2 \d w \le C \varrho^{-1} (E_0 + M_1/m_0)^2),
\end{align}
where we used that $|\bar{v}| \le   M_1/m_0$. Moreover, note that due to the definition of $\bar{v}$, it holds
\begin{align*}
\int_{\R^n} f(w) (w - \bar{v}) \cdot e \d w = 0\qquad\text{for all}\quad e\in \mathbb{S}^{n-1}.
\end{align*}
Therefore, using also \eqref{eq:center-of-mass-ass}, and \eqref{eq:center-of-mass-help-1}, we obtain
\begin{align*}
\lambda_1  \eta &\le \eta \int_{\{ (w - \bar{v})\cdot e_0 \ge \eta \}} f(w)\d w \\
&\le \int_{\{ (w - \bar{v})\cdot e_0  \ge 0 \}} f(w) (w - \bar{v}) \cdot e_0 \d w \\
&= - \int_{\{  (w - \bar{v})\cdot e_0  \le 0 \}} f(w) (w - \bar{v}) \cdot e_0 \d w \\
&\le \int_{\R^n \setminus B_\varrho(\bar{v})} f(w) |(w - \bar{v}) \cdot e_0| \d w + \int_{B_\varrho(\bar{v}) \cap \{   (w - \bar{v})\cdot e_0 \le 0 \}} f(w) |(w - \bar{v}) \cdot e_0| \d w \\
&\le C\varrho^{-1} (E_0 + (M_1/m_0)^2) + \varrho \int_{B_\varrho(\bar{v}) \cap \{ (w - \bar{v})\cdot e_0 \le 0 \}} f(w) \d w.
\end{align*}
Altogether, choosing $\varrho$ so large that $ C\varrho^{-1} (E_0 + (M_1/m_0)^2) \le \frac{\lambda_1 \eta}{2}$, we deduce
\begin{align*}
\int_{B_\varrho(\bar{v}) \cap \{   (w - \bar{v})\cdot e_0  \le 0 \}} f(w) \d w \ge \frac{\lambda_1 \eta}{2\varrho},
\end{align*}
as desired.
\end{proof}

As a consequence of the previous two lemmas, we are now in a position to prove \autoref{prop:subtract-tube}.

\begin{proof}[Proof of \autoref{prop:subtract-tube}]
First, note that since $|\bar v|\le M_1/m_0$, it suffices to establish the existence of $R,\delta,c>0$ such that
\begin{align}
\label{eq:subtract-tube-help-0}
\int_{B_R  \setminus L_{\delta}} f(\bar v + w) \d w \ge c.
\end{align}
To prove \eqref{eq:subtract-tube-help-0}, we fix $R$, $\delta_0$, and $\lambda$, to be the parameters from \autoref{lemma:mass-location}. Hence, for any $\sigma \in \mathbb{S}^{n-1}$:
\begin{align}
\label{eq:subtract-tube-help-1}
   \int_{B_{R}  \cap \{{\rm dist}(\cdot, \langle \sigma\rangle) \ge \delta_0\}} f(\bar v + w) \d w \ge \lambda.
\end{align}
Let us set $\delta = \frac{\delta_0}{2}$ and let $L_{\delta}$ be a fixed linear tube of radius $\delta$ along the line $L := \{a_0 + t e_0 : t\in \R\}\subset \R^n$, for some $a_0\in \R^n$, $e_0\in \mathbb{S}^{n-1}$. 

We choose $\sigma = e_0$. If $L_\delta \subset \{{\rm dist}(\cdot, \langle\sigma\rangle )< \delta_0 = 2\delta\}$ we are done by \eqref{eq:subtract-tube-help-1}, so let us assume  that 
\[
L_\delta \subset \{w\cdot e \ge \delta_0/2 = \delta\}\quad\text{for some}\quad e\perp \sigma,\ e\in \mathbb{S}^{n-1}.
\]
From \eqref{eq:subtract-tube-help-1} we can further assume 
\[
\int_{B_{R} \cap L_\delta} f( \bar v + w) \d w \ge \frac{\lambda}{2},
\]
since otherwise \eqref{eq:subtract-tube-help-0} follows with $c = \frac{\lambda}{2}$. In particular, we have 
\[
\int_{B_{R} \cap \{w\cdot e \ge \delta\}} f( \bar v + w) \d w \ge \frac{\lambda}{2},
\]
so that, by \autoref{lemma:center-of-mass} we obtain 
\[
\int_{B_{R} \cap \{w\cdot e \le 0\}} f( \bar v + w) \d w \ge \lambda_2,
\]
for some $\lambda_2 > 0$. The result now follows because $B_R\setminus L_\delta\supset B_R\cap \{w\cdot e \le 0\}$. 
\end{proof}

\section{Proof of nondegeneracy}
\label{sec:nondeg}

In this section, we establish the nondegeneracy of the Boltzmann kernel $\tilde{K}_f$ under the change of variables for any $v_0 \in \R^n$.

\begin{theorem}
\label{thm:change-of-variables-nondegeneracy}
Let $s\in (0, 1)$ and $\gamma \in (-n,\gamma_0]$ for some $\gamma_0 > -n$. Assume that $f$ is nonnegative and satisfies \eqref{eq:mass}, \eqref{eq:temperatureB}, and \eqref{eq:moment} for some $q > 2$. Then, the kernel $\tilde{K}_f$ given by \eqref{eq:Ktildedef} and \eqref{s-to-1-kernel} with $v_0\in \R^n$ satisfies 
\begin{align*}
\inf_{e \in \mathbb{S}^{n-1}} \int_{B_r} \tilde{K}_f(v,v+h) (h \cdot e)_+^2 \d h \ge  \lambda r^{2-2s} \qquad\text{for all}\quad  r > 0, ~~ v \in B_2,
\end{align*}
uniformly in $v_0$, with $\lambda > 0$ depending only on $n,m_0,M_0,p_0,M_q,q,$ and $\gamma_0$.
\end{theorem}

We split the proof into two parts, treating separately the cases $|v_0| \le 2$ and $|v_0| > 2$.

\subsection{Nondegeneracy near the origin}

First, we establish a nondegeneracy estimate which does not take into account the change of variables. This result will be used in the proof of \autoref{thm:change-of-variables-nondegeneracy} only in case $|v_0| \le 2$. Recall that for the bounds on $K_f$ we assume \eqref{s-to-1-kernel}.

\begin{proposition}
\label{thm:Boltzmann-kernel-nondegenerate}
Let $s\in (0, 1)$ and $\gamma \in (-n,\gamma_0]$ for some $\gamma_0 > -n$. Assume that $f$ is nonnegative and satisfies \eqref{eq:mass}, \eqref{eq:temperatureB}, and \eqref{eq:moment} for some $q  > 2$. Then, for every $r > 0$ and $v \in \R^n$:
\begin{align}
\label{eq:nondegenerate}
\inf_{e \in \mathbb{S}^{n-1}} \int_{B_r} K_f(v,v+h) (h \cdot e)_+^2 \d h \ge  \lambda(v) r^{2-2s},
\end{align}
where $\lambda(v) \ge c (1+|v|)^{\gamma+2s-2}$ for some $c>0$ depending only on $n, m_0, M_0, p_0, M_q$, $q$, and $\gamma_0$.
\end{proposition}

First, let us rewrite $K_f(v,v+h) = K_f(v;h)$ and deduce from \eqref{eq:K-bounds}-\eqref{s-to-1-kernel} that $K_f(v;\cdot)$ is comparable to a homogeneous kernel for any $v$, so that we can write  
\begin{equation}
\label{eq:Kf_a}
K_f(v;h) \asymp (1-s)|h|^{-n-2s} a(v;h/|h|)\quad\text{where}\quad a(v;\theta) = \int_{w \perp \theta} f(v + w) |w|^{\gamma + 2s + 1} \d w.
\end{equation}
Therefore, using polar coordinates and also the symmetry $a(v,\theta) = a(v,-\theta)$, the nondegeneracy condition \eqref{eq:nondegenerate} can be equivalently reformulated as follows: Check that for any $v \in \R^n$ and any $e \in \mathbb{S}^{n-1}$ it holds
\begin{align}
\label{eq:hom-nondegenerate}
\int_{\mathbb{S}^{n-1}} a(v;\theta) |\theta \cdot e|^2 \d \theta  = 2 \int_{\mathbb{S}^{n-1}} a(v;\theta) (\theta \cdot e)_+^2 \d \theta  \ge \lambda(v),
\end{align}
for some $\lambda(v)$ comparable to $(1+|v|)^{\gamma+2s-2}$. The resulting constant in \eqref{eq:nondegenerate} does not depend on $s$ due to the following identity: $(1-s) \int_{0}^r |h|^{-1-2s+2}\, dh = r^{2-2s}$. 

\begin{remark}
Note that the left-hand side in \eqref{eq:hom-nondegenerate} vanishes if and only if there exist $v \in \R^n$ and $e \in \mathbb{S}^{n-1}$ such that for any $\theta \not\perp e$ it holds 
\begin{align*}
f(v+w) \equiv 0 ~~ \quad\text{for a.e.}\  w \perp \theta,
\end{align*}
or, in other words, if $f$ is supported on a line. 
\end{remark}

The following is a standard calculus identity, it can be found in \cite{Sil16} (see also \cite[Lemma A.10]{ImSi20}).
\begin{align*}
\int_{\mathbb{S}^{n-1}} \int_{w \perp \theta} g(w) \d w \d \theta = \int_{\R^n} g(z) |z|^{-1} \d z.
\end{align*}

We will make use of the following weighted version of such identity:
\begin{align}
\label{eq:weighted-trafo}
\int_{\mathbb{S}^{n-1}} \left( \int_{w \perp \theta} g(w,\theta) \d w \right) \d \theta = \int_{\R^n} \left( \int_{\theta \perp z} g(z,\theta) \d \theta \right) |z|^{-1} \d z.
\end{align}

With the help of this identity, we can derive the following equivalent version of the nondegeneracy condition \eqref{eq:nondegenerate}:

\begin{lemma}
\label{lemma:Gineq}
The inequality \eqref{eq:nondegenerate} is equivalent to the following condition: 

For every $v \in \R^n$ and any $e \in \mathbb{S}^{n-1}$ it holds
\begin{align}
\label{eq:nondegenerate-rewritten}
\int_{\R^n} f(v+z) G(z,e) |z|^{\gamma + 2s} \d z \ge \lambda(v),
\end{align}
where $\lambda(v) \asymp (1+|v|)^{\gamma + 2s - 2}$, and 
\begin{align*}
G(z,e) = \int_{\mathbb{S}^{n-1} \cap \{\theta \perp z\} } |\theta \cdot e|^2 \d \theta =  \int_{\mathbb{S}^{n-1} \cap \{\theta \perp z\} } \cos^2(\theta,e) \d \theta = \int_{\mathbb{S}^{n-1} \cap \{\theta \perp z\} } \sin^2(z,e) \d \theta = c(n) \sin^2(z,e),
\end{align*}
where $\cos(\theta, e)$ and $\sin(\theta, e)$ denote respectively the cosine and sine of the angle between $\theta$ and $e$. 
\end{lemma}

\begin{proof}
We apply \eqref{eq:weighted-trafo} with  $g(w,\theta) = f(v+w)|w|^{\gamma + 2s + 1} |\theta \cdot e|^2$. Then, the desired result follows from the following computation:
\begin{align*}
\int_{\mathbb{S}^{n-1}} a(v;\theta) |\theta \cdot e|^2 \d \theta &= \int_{\mathbb{S}^{n-1}} \left( \int_{w \perp \theta} f(v+w) |w|^{\gamma + 2s +1}  |\theta \cdot e|^2 \d w \right) \d \theta \\
&= \int_{\R^n} f(v+z) \left(\int_{\theta \perp z} |\theta \cdot e|^2 \d \theta \right) |z|^{\gamma + 2s} \d z \\
&= \int_{\R^n} f(v+z) G(z,e) |z|^{\gamma + 2s} \d z,
\end{align*}
together with \eqref{eq:Kf_a}-\eqref{eq:hom-nondegenerate}. 
\end{proof}

We are now in a position to prove \autoref{thm:Boltzmann-kernel-nondegenerate}:

\begin{proof}[Proof of \autoref{thm:Boltzmann-kernel-nondegenerate}]
Recall that by \eqref{eq:nondegenerate-rewritten}, it suffices to prove that for any $v \in \R^n$ and any $e \in \mathbb{S}^{n-1}$ it holds
\begin{align}
\label{eq:nondegenerate-goal}
\int_{\R^n} f(w) \sin^2(w - v,e) |w - v|^{\gamma + 2s} \d w \ge \lambda(v),
\end{align}
where $\lambda(v) \asymp (1 + |v- \bar{v}|)^{\gamma+2s-2}$ (using also that $|\bar v |\le M_1/m_0$).
Let us fix $v \in \R^n$ and $e \in \mathbb{S}^{n-1}$. Let us denote $v + \R e = \{v + t e  : t \in \R \}$. Moreover, let $R,\delta, c$, be the constants from \autoref{prop:subtract-tube}, and let $L_\delta$ denote the tube of radius $\delta$ around $v+\R e$. 

  We claim that there exists $\tilde c > 0$, depending only on $R,\delta, c$, such that 
\begin{align}
\label{eq:G-estimate-case-1}
\inf_{w \in B_R(\bar{v})\setminus L_\delta } \sin^2(w - v,e) \ge \tilde c (1 + |v - \bar{v}|)^{-2}.
\end{align}

This follows because, by the definition of $L_{\delta}$, 
\begin{align*}
\inf_{w \in B_R(\bar{v}) \setminus L_\delta } \dist(w,v + \R e) \ge \delta, 
\end{align*}
and hence, for any $w \in B_R(\bar{v}) \setminus L_\delta$:
\begin{equation}
\label{eq:Gcomp}
\sin^2(w-v,e) = \frac{\dist(w,v + \R e)^2}{|v - w|^2} \ge \frac{\delta^2}{2R^2 + 2|v - \bar{v}|^2},
\end{equation}
where we have also used that
\[
|w-v|^2\le 2 |w-\bar v|^2 + 2|v-\bar v|^2\le 2 R^2+ 2 |v-\bar v|^2. 
\]
This yields  \eqref{eq:G-estimate-case-1}. Moreover, note that for some $c' > 0$, depending only on $R,\delta$, we have
\begin{align*}
\inf_{w \in B_R(\bar{v}) \setminus L_\delta } |w - v| \ge \max \left( \delta , |v-\bar{v}| - R \right) \ge c' (1 + |v - \bar{v}|).
\end{align*}
Consequently, we deduce from \autoref{prop:subtract-tube} and the previous inequalities:
\begin{equation}
\label{eq:eq-case-1_Br}
\begin{split}
\int_{\R^n} f(w) \sin^2(w - v,e) |w - v|^{\gamma + 2s} \d w &\ge  \int_{B_R(\bar{v})\setminus L_\delta}f(w)  \sin^2(w - v,e) |w - v|^{\gamma + 2s} \d w \\
&\ge \bar c (1 + |v - \bar{v}|)^{\gamma + 2s -2} \int_{B_R(\bar{v})\setminus L_\delta} f(w) \d w\\
&\ge \bar c c  (1 + |v - \bar{v}|)^{\gamma + 2s -2}
\end{split}
\end{equation}
for some $c,\bar c > 0$, depending only on $m_0,M_0,p_0,M_q,q$, 
as desired.
\end{proof}

\subsection{Nondegeneracy under change of variables}

In this section we prove the nondegeneracy condition for $|v_0| > 2$, thereby concluding the proof of \autoref{thm:change-of-variables-nondegeneracy}.

\begin{proof}[Proof of \autoref{thm:change-of-variables-nondegeneracy}]
The proof is complete once we can show the following property: 

For any $e \in \mathbb{S}^{n-1}$ and any $v_0 \in \R^n$, and $v \in B_2$:
\begin{align*}
\int_{B_r(v)} \tilde{K}_f(v,v') |(v'-v) \cdot e|^2 \d v' \ge  r^{2-2s} \lambda\quad\text{for all}\quad r > 0. 
\end{align*}
We fix any $r > 0$. Note that we only need to consider the case $|v_0| \ge 2$ due to \autoref{thm:Boltzmann-kernel-nondegenerate}, and the definition of $\mathcal{T}_0$ in case $|v_0| < 2$. Indeed, if $|v_0|<2$ we have $\tilde K_f(v, v') =  K_f(v_0+v, v_0+v')$, and by \autoref{thm:Boltzmann-kernel-nondegenerate}
\begin{align*}
\int_{B_r(v)} \tilde{K}_f(v,v') |(v'-v) \cdot e|^2 \d v' & =   \int_{B_r(v)} {K}_f(v_0+v,v_0+v') |(v'-v) \cdot e|^2 \d v'\\ & \ge c   (1+|v_0+v|)^{\gamma+2s-2} r^{2-2s}\ge  \lambda r^{2-2s},
\end{align*}
for $|v| < 2$ and $|v_0|< 2$. Let us therefore assume $|v_0|\ge 2$. We divide the proof into three steps.

{\bf Step 1:} Using the definition of $\tilde{K}_f(v,v')$ and writing $\tilde{v} = v_0 + \tau_0 (v)$ and $\tilde{h} = \tau_0(v '- v)$, we rewrite
\begin{align*}
\int_{B_r(v)} \tilde{K}_f(v,v') & |(v'-v) \cdot e|^2 \d v' \\
&= |v_0|^{-1-\gamma-2s} \int_{B_r(v)} K_f(v_0 + \tau_0 (v), v_0 + \tau_0 (v')) |(v'-v) \cdot e|^2 \d v'\\
&= |v_0|^{-1-\gamma-2s} \int_{E_r} K_f(\tilde{v} , \tilde{v} + \tilde{h}) |\tau_0^{-1} (\tilde{h}) \cdot e|^2 |\det \tau_0|^{-1} \d \tilde{h} \\
&= |v_0|^{-\gamma-2s} \int_{E_r}  K_f(\tilde{v} , \tilde{v} + \tilde{h}) |\tau_0^{-1}( \tilde{h}) \cdot e|^2  \d \tilde{h} \\
&\asymp (1-s) |v_0|^{-\gamma-2s} \int_{E_r} |\tilde{h}|^{-n-2s} \left( \int_{\tilde{w} \perp \tilde{h} } f(\tilde{v} + \tilde{w}) |\tilde{w}|^{\gamma+2s+1} \d \tilde{w} \right) |\tau_0^{-1} (\tilde{h}) \cdot e|^2 \d \tilde{h},
\end{align*}
where in the last step we have used \eqref{eq:Kf_a}, and where $E_r = E_r(0)$ is the ellipsoid centered at the origin with side length $r/|v_0|$ in the $v_0$-direction, and side length $r$ in all directions perpendicular to $v_0$.

Next, observe that we have the following generalization of \eqref{eq:weighted-trafo} (see \cite[eq. (5.9)]{ImSi22}):
\begin{align}
\label{eq:weighted-trafo2}
\int_{E_r} \left( \int_{\R^n \cap \{ \tilde{w} \perp \tilde{h} \} } g(\tilde{w},\tilde{h}) \d \tilde{w} \right) \d \tilde{h} = \int_{\R^n} \left( \int_{E_r \cap \{\tilde{h} \perp \tilde{w} \} } g(\tilde{w},\tilde{h}) \frac{|\tilde{h}|}{|\tilde{w}|} \d \tilde{h} \right) \d \tilde{w}.
\end{align}

An application of \eqref{eq:weighted-trafo2} with $g(\tilde{w},\tilde{h}) = |\tilde{h}|^{-n-2s} f(\tilde{v} + \tilde{w}) |\tilde{w}|^{\gamma + 2s + 1} |\tau_0^{-1} (\tilde{h}) \cdot e|^2$ yields
\begin{align*}
\frac{1}{1-s} \int_{B_r(v)} \tilde{K}_f(v,v') & |(v'-v) \cdot e|^2 \d v'\\
 &\asymp |v_0|^{-\gamma-2s} \int_{\R^n} \left( \int_{E_r \cap \{\tilde{h} \perp \tilde{w} \} } |\tilde{h}|^{-n-2s} f(\tilde{v} + \tilde{w}) |\tilde{w}|^{\gamma + 2s + 1} |\tau_0^{-1} (\tilde{h}) \cdot e|^2 \frac{|\tilde{h}|}{|\tilde{w}|} \d \tilde{h} \right)  \d \tilde{w} \\
&= |v_0|^{-\gamma-2s} \int_{\R^n} f(\tilde{v} + \tilde{w}) \left( \int_{E_r \cap \{\tilde{h} \perp \tilde{w} \} } |\tilde{h}|^{-n-2s+1} |\tau_0^{-1} (\tilde{h}) \cdot e|^2 \d \tilde{h} \right) |\tilde{w}|^{\gamma + 2s} \d \tilde{w} \\
&= \int_{\R^n} f(\tilde{v} + \tilde{w}) \tilde G(\tilde{w},e) |\tilde{w}|^{\gamma + 2s} \d \tilde{w}\\
&= \int_{\R^n} f(\tilde{w}) \tilde G(\tilde{w} - \tilde{v},e) |\tilde{w} - \tilde{v}|^{\gamma + 2s} \d \tilde{w},
\end{align*}
where
\begin{align*}
\tilde G(\tilde{w} - \tilde{v},e) &:= |v_0|^{-\gamma-2s} \int_{E_r \cap \{\tilde{h} \perp (\tilde{w} - \tilde{v}) \} } |\tilde{h}|^{-n-2s+1} |\tau_0^{-1} (\tilde{h}) \cdot e|^2 \d \tilde{h}.
\end{align*}

We notice that for any two vectors $a,b \in \R^n$ it holds that $a \cdot b = \tau_0(a) \cdot \tau_0^{-1}(b)$. Indeed, let us assume without loss of generality (up to a coordinate transform) that $v_0/|v_0| = e_1$. Then,
\begin{align*}
a \cdot b = \sum_{i = 1}^n a_i b_i =  \frac{a_1}{|v_0|} (b_1 |v_0|) + \sum_{i = 2}^n a_i b_i = \tau_0(a) \cdot \tau_0^{-1}(b).
\end{align*}

Therefore, we can compute
\begin{align*}
\tilde G(\tilde{w} - \tilde{v},e) &= |v_0|^{-\gamma-2s} \int_{E_r \cap \{\tilde{h} \perp (\tilde{w} - \tilde{v}) \} } |\tilde{h}|^{-n-2s+1} |\tilde{h} \cdot \tau_0^{-1}(e)|^2 \d \tilde{h} \\
&= |v_0|^{-\gamma-2s} \int_{E_r \cap \{\tilde{h} \perp (\tilde{w} - \tilde{v}) \} } |\tilde{h}|^{-n -2s + 3} |\tau_0^{-1}(e)|^2 \cos^2(\tilde{h} , \tau_0^{-1}(e)) \d \tilde{h} \\
&= \sin^2((\tilde{w} - \tilde{v}) , \tau_0^{-1}(e)) |\tau_0^{-1}(e)|^2 |v_0|^{-\gamma-2s} \int_{E_r \cap \{\tilde{h} \perp (\tilde{w} - \tilde{v}) \} } |\tilde{h}|^{-(n-1) - 2s + 2}  \d \tilde{h}.
\end{align*}

{\bf Step 2:} Our next goal is to estimate the terms in $\tilde G(\tilde{w}-\tilde{v},e)$ separately. We will do so only for $\tilde{w} \in B_R(\bar{v})$, where $R$ is the constant from  \autoref{prop:subtract-tube}.\\
First, we claim that
\begin{align}
\label{eq:ellipse-integral}
\int_{E_r \cap \{\tilde{h} \perp (\tilde{w} - \tilde{v}) \} } |\tilde{h}|^{-(n-1)-2s+2} \d \tilde{h} \ge \frac{c}{1-s} r^{2-2s}\quad\text{for}\quad \tilde w\in B_R(\bar v). 
\end{align}
In fact, according to \cite[(5.10)]{ImSi22}, we have that the set $E_r \cap \{\tilde{h} \perp (\tilde{w} - \tilde{v}) \}$ is an $(n-1)$-dimensional ellipsoid whose smallest radius is bounded below by $\rho$,
\begin{align*}
\rho := r \left(|v_0|^2\sin^2(v_0,\tilde{w} - \tilde{v}) + \cos^2(v_0,\tilde{w}-\tilde{v}) \right)^{-\frac{1}{2}}.
\end{align*}
Hence, we obtain
\begin{align*}
\int_{E_r \cap \{\tilde{h} \perp (\tilde{w} - \tilde{v}) \} } |\tilde{h}|^{-(n-1)-2s+2} \d \tilde{h} \ge \frac{c}{1-s} \rho^{2-2s} \ge \frac{c}{1-s} r^{2-2s} \left(|v_0|^2\sin^2(v_0,\tilde{w} - \tilde{v}) + \cos^2(v_0,\tilde{w}-\tilde{v}) \right)^{s - 1}.
\end{align*}

Note that in case $|v_0| \le  10(R + |\bar{v}|+1)\le C$, the claim  \eqref{eq:ellipse-integral} follows trivially. In case $|v_0| \ge  10(R + |\bar{v}|+1)$, we argue as follows:
Since $\cos^2(v_0,\tilde{w} - \tilde{v}) \le 1$, it is enough to estimate
\begin{align*}
|v_0|^2\sin^2(v_0,v_0 - z) \le C\quad\text{for}\quad z = \tilde{w} - \tau_0 (v)\quad\text{and}\quad \tilde w\in B_R(\bar v).
\end{align*}
Observe that $z \in B_{2R}(\bar{v})$, since $|\tau_0 (v)|\le |v|\le  2\le R$. This property is now satisfied since the condition $|v_0| \ge 10(R + |\bar{v}| + 1)$ implies that 
\begin{align*}
\sin^2(v_0,v_0 - z) \le \frac{|z|^2}{|v_0 - z|^2} \le \frac{|z|^2}{(|v_0| - |z|)^2} \le \frac{(|\bar{v}| + 2R)^2}{(|v_0| - (|\bar{v}| + 2R))^2} \le C|v_0|^{-2}.
\end{align*}
This proves \eqref{eq:ellipse-integral}. 

Next, we note that
\begin{align}
\label{eq:t0-norm-compute}
|\tau_0^{-1}(e)|^2 = 1 + (|v_0|^2 - 1)\frac{|v_0 \cdot e|^2}{|v_0|^2} = 1 + (|v_0|^2 - 1)\cos^2(v_0,e).
\end{align}

Therefore, we get for any $\tilde{w} \in B_R(\bar{v})$:
\begin{align*}
\tilde G(\tilde{w} - \tilde{v},e) \ge c |v_0|^{-\gamma-2s} r^{2-2s} \sin^2((\tilde{w} - \tilde{v}) , \tau_0^{-1}(e)) [1 + (|v_0|^2 - 1)\cos^2(v_0,e)].
\end{align*}

We now combine all the aforementioned estimates. This yields:
\begin{align*}
\int_{B_r(v)} & \tilde{K}_f(v,v') |(v'-v) \cdot e|^2 \d v' \\
&= (1-s) \int_{\R^n} f(\tilde{w}) \tilde G(\tilde{w} - \tilde{v},e) |\tilde{w} - \tilde{v}|^{\gamma + 2s} \d \tilde{w}\\
&\ge c r^{2-2s} |v_0|^{-\gamma-2s}  [1 + (|v_0|^2 - 1)\cos^2(v_0,e)]   \left[\int_{B_R(\bar{v})} f(\tilde{w}) \sin^2((\tilde{w} - \tilde{v}) , \tau_0^{-1}(e)) |\tilde{w} - \tilde{v}|^{\gamma + 2s} \d \tilde{w} \right].
\end{align*}

{\bf Step 3:} In order to conclude the proof, let us first consider, as before, the case $2 \le |v_0| \le  10(R + |\bar{v}|+1)$. In this case, we apply \eqref{eq:eq-case-1_Br} with unit vector $\tau_0^{-1}(e)/|\tau_0^{-1}(e)|$ and obtain
\begin{align*}
\int_{B_r(v)} \tilde{K}_f(v,v') |(v'-v) \cdot e|^2 \d v' &\ge c r^{2-2s} \left[\int_{B_R(\bar{v})} f(\tilde{w}) \sin^2((\tilde{w} - \tilde{v}) , \tau_0^{-1}(e)) |\tilde{w} - \tilde{v}|^{\gamma + 2s} \d \tilde{w} \right] \\
&\ge c (1+|\tilde v - \bar v |)^{\gamma+2s-2} r^{2-2s} \ge c r^{2-2s}.
\end{align*}
We have also used here that $|\tilde v|+|\bar v|\le C$ when $|v_0|\le 10(R+|\bar v|+1)$.

Let us suppose now that $|v_0| \ge 10(R + |\bar{v}|+1)$. First, we observe that since $v \in B_2$, it holds $\tilde{v} \in E_2(v_0) \subset B_2(v_0)$, and therefore for any $\tilde{w} \in  B_R(\bar{v})$:
\begin{align*}
|\tilde{w} - \tilde{v}|^{\gamma + 2s} \asymp |v_0 - \bar{v}|^{\gamma + 2s} \ge c |v_0|^{\gamma + 2s},
\end{align*}
where we   used that $|v_0| \ge 10(R + |\bar{v}|+1)$.
Thus, it remains to verify the following property:
\begin{align}
\label{eq:trafo-nondegeneracy-help-0}
\left( \int_{B_R(\bar{v})} f(\tilde{w}) \sin^2((\tilde{w} - \tilde{v}) , \tau_0^{-1}(e)) \d \tilde{w} \right) [1 + (|v_0|^2 - 1)\cos^2(v_0,e)] \ge c>0.
\end{align}

Observe that, the result holds true depending on $c_0$ once $\cos^2(v_0, e) \ge c_0 > 0$ for some $c_0 >0$. Indeed, thanks to \eqref{eq:Gcomp} and \autoref{prop:subtract-tube}, we can proceed exactly as in \eqref{eq:eq-case-1_Br}, replacing again $e$ by $\tau_0^{-1}(e)/|\tau_0^{-1}(e)|$, to deduce
\[
\int_{B_R(\bar{v})} f(\tilde{w}) \sin^2((\tilde{w} - \tilde{v}) , \tau_0^{-1}(e)) \d \tilde{w}\ge \bar c (1+|\tilde v - \bar v|)^{-2}\ge c(1+|v_0|)^{-2}, 
\]
and so \eqref{eq:trafo-nondegeneracy-help-0} holds whenever $\cos^2(v_0, e)\ge c_0 > 0$.

Let us fix 
\[
c_0 = \frac{1}{10(R+|\bar v|+2)^2  },
\]
and prove that \eqref{eq:trafo-nondegeneracy-help-0} also holds in the case $\cos^2(v_0, e)\le c_0$. \\
We start by noticing that on the one hand, by the triangle inequality, and since $\tilde w\in B_R(\bar v)$ and $|\tau_0 (v)|\le 2$, 
\begin{equation}
\label{eq:combinedwith_sin2}
\begin{split}
\sin^2((\tilde{w} - \tilde{v}) , \tau_0^{-1}(e)) &= \frac{\dist(\tilde{v} + \tau_0^{-1}(e)\R, \tilde{w})^2}{ |\tilde{v} - \tilde{w}|^2} \\
&\ge c \frac{\inf_{\tau  \in \R}|v_0 + \tau_0^{-1}(e)\tau|^2-2(R+|\bar v|+2)^2}{|v_0|^2},
\end{split}
\end{equation}
where we have also used $|\tilde v - \tilde w|^2 \le C (R^2+|v_0|^2  + |\bar v|^2 + 4) \le C|v_0|^2$. On the other hand, denoting for the sake of readability $\eta^2 := \cos^2(v_0, e) = \frac{(v_0\cdot e)^2}{|v_0|^2}\le c_0$, and using \eqref{eq:t0-norm-compute} we have
\begin{align*}
\frac{\inf_{\tau  \in \R}|v_0 + \tau_0^{-1}(e)\tau|^2}{|v_0|^2} &= \sin^2(v_0, \tau_0^{-1}(e)) = 1- \frac{(v_0\cdot \tau_0^{-1}(e))^2}{|v_0|^2|\tau_0^{-1}(e)|^2} \\
&= 1- \frac{(\tau_0^{-1} (v_0)\cdot e)^2}{{|v_0|^2|\tau_0^{-1}(e)|^2}} = 1 - \frac{ (v_0\cdot e)^2}{|\tau_0^{-1}(e)|^2}\\
&= 1 - \frac{|v_0|^2\eta^2}{1-\eta^2 + \eta^2|v_0|^2} = \frac{1-\eta^2}{1-\eta^2 + \eta^2|v_0|^2}.
\end{align*}

Thus, in order to verify \eqref{eq:trafo-nondegeneracy-help-0} it is enough to check (since $B_r(\bar v)$ always contains mass, \autoref{lemma:mass-location})
\[
\left( \frac{1-\eta^2}{1+(|v_0|^2-1)\eta^2} -  2\frac{(R+|\bar v|+2)^2}{|v_0|^2}\right) [1 + (|v_0|^2 - 1)\eta^2]\ge c > 0,
\]
which holds true in particular if we can verify.
\[
 1-\eta^2 -  2\frac{(R+|\bar v|+2)^2}{|v_0|^2} - 2(R+|\bar v |+2)^2 \eta^2\ge c > 0.
\]
Since $\eta^2 \le c_0\le \frac{1}{10(R+|\bar  v|  +2)^2}$ and $|v_0|\ge 10 (R+|\bar v|+1)$, the latter inequality (and therefore also the former inequality) holds true, and the proof is complete. 
\end{proof}

\section{Proof of coercivity}
\label{sec:coercive}

In this section, we prove that the nonlocal energy induced by the Boltzmann equation is coercive and that the coercivity constants do not degenerate under the change of variables for any $v_0 \in \R^n$.

\begin{theorem}
\label{thm:change-of-variables-coercivity}
Let $s\in (0, 1)$ and $\gamma \in (-n,\gamma_0]$ for some $\gamma_0 > -n$.  Assume that $f$ is nonnegative and satisfies \eqref{eq:mass}, \eqref{eq:temperatureB}, and \eqref{eq:moment} for some $q > 2$. Then, the kernel $\tilde{K}_f$ given by \eqref{eq:Ktildedef} and \eqref{s-to-1-kernel} with $v_0\in \R^n$  satisfies the following property uniformly in $v_0$:\\
For any $g$ supported in $B_2$ it holds
\begin{align*}
\int_{B_2} \int_{\R^n} (g(v) - g(v'))^2 \tilde{K}_f(v,v') \d v \d v' \ge   \lambda [g]_{H^s(\R^n)}^2 - \Lambda \Vert g \Vert_{L^2(\R^n)}^2
\end{align*}
with constants $\lambda, \Lambda > 0$, depending only on $n,m_0,M_0,p_0,M_q, q$, and $\gamma_0$.
\end{theorem}

We recall that the fractional Sobolev seminorm $[\cdot]_{H^s(\R^n)}$ is given by 
\[
[g]_{H^s(\R^n)} = \frac{c_{n, s}}{2}\int_{\R^n}\int_{\R^n}\frac{|g(x) -g(y)|^2}{|x-y|^{n+2s}} \d x \d y = \int_{\R^n} |(-\Delta)^{s/2} g|^2 = \int_{\R^n}g(-\Delta)^s g,
\]
where for us, it is important to notice that $c_{n, s} \asymp (1-s)$ as $s\uparrow 1$. 

Our proof is a direct consequence of the following result, which was obtained in \cite{GrSt11}:

\begin{proposition}
\label{prop:GrSt}
Let $s\in (0, 1)$ and $\gamma \in (-n, \gamma_0]$ for some $\gamma_0 > -n$. Assume that $f$ is nonnegative and satisfies \eqref{eq:mass}, \eqref{eq:temperatureB}, and \eqref{eq:moment} for some $q > 2$. Then, for any $g$ it holds
\begin{align}
\label{eq:GrSt-estimate}
\begin{split}
\int_{\R^n} &\int_{\R^n} (g(v) - g(v'))^2 K_f(v,v') \d v \d v' \\
&\ge \lambda \int_{\R^n} \int_{\R^n} (g(v) - g(v'))^2 [(1+|v|^2)(1+|v'|^2)]^{\frac{\gamma+2s+1}{4}} \frac{1-s}{d(v,v')^{n+2s}} \1_{\{ d(v,v') \le 1 \}}(v,v') \d v \d v',
\end{split}
\end{align}
where $\lambda > 0$ depends only on $n,m_0,M_0,p_0,M_q,q$, and  $\gamma_0$, and 
\begin{align*}
d(w,w') := \sqrt{ \left|w-w' \right|^2 + \frac{1}{4} \left(|w|^2 - |w'|^2\right)^2} ~~ \forall w,w'\in \R^n.
\end{align*}
\end{proposition}

\begin{proof}
In \cite[(11) in Theorem 1]{GrSt11} the authors establish the following estimate\footnote{The result in \cite{GrSt11} does not keep track of the dependence on $s$. A quick inspection of the proof, however, shows that if the kernel $B$ from \eqref{eq:Bdef} is multiplied by a constant, this applies as well to the constant from \cite[(11) in Theorem 1]{GrSt11} (in the proof, the $s$-dependence becomes apparent only in \cite[eq. (42)]{GrSt11}).}
\begin{align*}
N_f(g) \ge \lambda  \int_{\R^n} \int_{\R^n} (g(v) - g(v'))^2 [(1+|v|^2)(1+|v'|^2)]^{\frac{\gamma+2s+1}{4}} \frac{1-s}{d(v,v')^{n+2s}} \1_{\{ d(v,v') \le 1 \}}(v,v') \d v \d v'
\end{align*}
under the assumption that there exist $R > \delta > 0$ and $c_1 > 0$, such that 
\begin{align}
\label{eq:GrSt-ass}
\int_{B_R \setminus L_\delta} f(v) \d v \ge c_1,
\end{align}
where $L_\delta$ is any linear tube of radius $\delta$. Here, $N_f(g)$ is defined as follows
\begin{align*}
N_f(g) := \int_{\R^n} \int_{\R^n} \int_{\mathbb{S}^{n-1}} (g(v) - g(v'))^2 f(v_{\ast}) B(|v-v_{\ast}|,\sigma) \d \sigma \d v_{\ast} \d v,
\end{align*}
where $v'= \frac{v + v_{\ast}}{2} + \frac{|v - v_{\ast}|}{2} \sigma$ is as in \eqref{eq:v-v-prime}.

In \cite[(11) in Theorem 1]{GrSt11}, the constant $\lambda > 0$, depends only on $n, q, \gamma,\delta,R,c_1,M_0$. Since we assume that \eqref{eq:mass}, \eqref{eq:temperatureB}, and \eqref{eq:moment} for some $q > 2$ are satisfied, \eqref{eq:GrSt-ass} follows immediately by application of \autoref{prop:subtract-tube} with $c_1$ depending only on $m_0,M_0,p_0,M_q,q$.
Finally, we claim that
\begin{align}
\label{eq:Nfg-energy}
N_f(g) = \int_{\R^n} \int_{\R^n} (g(v) - g(v'))^2 K_f(v,v') \d v \d v'.
\end{align}
Clearly, once \eqref{eq:Nfg-energy} is established, the proof is complete. To prove \eqref{eq:Nfg-energy}, we rewrite $N_f(g)$ using Carleman coordinates, i.e., we set $w := v_{\ast}'$ and reparametrize the integration in $\sigma, v_{\ast}$ from the definition of $N_f(g)$ by $w,v'$ (see also \cite[Section 2.3]{ImSi20} and \cite[Lemma A.1]{Sil16}). This yields by the definition of $K_f(v,v')$ from \eqref{eq:Kf-def} and since under this transformation we have $v_{\ast} = v' + w$ and $w \perp v '- v$, and therefore $|v - v_{\ast}|^2 = |v-v' + w|^2 = |v-v'|^2 + |w|^2$:
\begin{align*}
N_f(g) &= \int_{\R^n} \int_{\R^n} (g(v) - g(v'))^2 \left( \frac{2^{n-1}}{|v'- v|} \int_{w \perp v'-v}  f(v' + w) B(r,\cos \theta) r^{-n+2} \d w \right) \d v' \d v \\
&= \int_{\R^n} \int_{\R^n} (g(v) - g(v'))^2 K_f(v',v) \d v' \d v.
\end{align*}
In the last line, we used that $\cos \theta$ and $r$ (see \eqref{eq:cos-r-def} for their definitions) remain invariant when the roles of $v$ and $v'$ are swapped. The proof of \eqref{eq:Nfg-energy} is complete.
\end{proof}

We are now in a position to give the proof of \autoref{thm:change-of-variables-coercivity}.

\begin{proof}[Proof of \autoref{thm:change-of-variables-coercivity}]
First, we assume that $|v_0| \le 2$. In that case, we have $\tilde K_f(t, x, v, v') = K_f(t_0+t, x_0+x+tv_0, v_0+v, v_0+v')$ and deduce from \autoref{prop:GrSt}  
\begin{align*}
\int_{B_2} & \int_{\R^n} (g(v) - g(v'))^2 \tilde{K}_f(v,v') \d v \d v' \\
&\ge \frac{1}{2}\int_{\R^n} \int_{\R^n} (g(v) - g(v'))^2 \tilde K_f(v,v') \d v \d v' \\
&= \frac12 \int_{\R^n} \int_{\R^n} (g(v-v_0) - g(v'-v_0))^2 K_f(v,v') \d v \d v' \\
&\ge c \int_{\R^n} \int_{\R^n} (g(v-v_0) - g(v'-v_0))^2 [(1+|v|^2)(1+|v'|^2)]^{\frac{\gamma+2s+1}{4}} \frac{1-s}{d(v,v')^{n+2s}} \1_{\{ d(v,v') \le 1 \}}(v,v') \d v \d v' \\
&\ge c (1-s) \int_{B_4} \int_{\R^n} (g(v-v_0) - g(v'-v_0))^2 |v-v'|^{-n-2s} \1_{\{ |v-v'| \le 1/6 \}}(v,v') \d v \d v'\\
&\ge c  [g]^2_{H^s(\R^n)} - c (1-s) \iint_{\{|v-v'|\ge 1/6 \}} (g(v-v_0) - g(v'-v_0))^2 |v-v'|^{-n-2s}  \d v \d v',
\end{align*}
where we used that 
\begin{align*}
d(v,v') \le |v - v'| + \frac{1}{2} |v-v'|(|v| + |v'|), 
\end{align*}
and that for $v \in B_4$ and $v' \in \R^n$ with $d(v,v') \le 1$, $|v'|\le 5$ and 
\begin{align*}
d(v,v') \le |v - v'| + \frac{1}{2} |v-v'|(|v| + |v'|)\le 6 |v-v'|. 
\end{align*}
Moreover, we also have
\begin{align*}
& \iint_{\{|v-v'|\ge 1/6\}}  (g(v-v_0) - g(v'-v_0))^2 |v-v'|^{-n-2s}   \d v \d v' \\
&\quad \le 4 \int_{\R^n} |g(v-v_0)|^2 \left( \int_{\R^n \setminus B_{1/6}(v)} |v-v'|^{-n-2s} \d v' \right) \d v \le \frac{C}{1-s} \Vert g \Vert_{L^2(B_2)}^2.
\end{align*}
Thus, by combination of the previous two estimates, we immediately deduce the desired result.

It remains to consider the case $|v_0| > 2$. We introduce the variables $\tilde{v} = v_0 + \tau_0 (v)$ and $\tilde{v}' = v_0 + \tau_0 (v')$ and define $\tilde{g}(\tilde{v}) = g(v)$. Then, we compute by transforming the integral twice and applying \autoref{prop:GrSt} to $\tilde{g}$:
\begin{align*}
\int_{B_2} & \int_{\R^n} (g(v) - g(v'))^2 \tilde{K}_f(v,v') \d v \d v' \\
& \ge \frac{1}{2}|v_0|^{-1-2s-\gamma}\int_{\R^n} \int_{\R^n} (g(v) - g(v'))^2 K_f(\tilde{v},\tilde{v}') \d v \d v' \\
&= \frac{1}{2}|v_0|^{1-2s-\gamma}\int_{\R^n} \int_{\R^n} (\tilde{g}(\tilde{v}) - \tilde{g}(\tilde{v}'))^2 K_f(\tilde{v},\tilde{v}') \d \tilde{v} \d \tilde{v}' \\
&\ge c |v_0|^{1-2s-\gamma} \int_{\R^n} \int_{\R^n} (\tilde{g}(\tilde{v}) - \tilde{g}(\tilde{v}'))^2 [(1+|\tilde{v}|^2)(1+|\tilde{v}'|^2)]^{\frac{\gamma+2s+1}{4}} \frac{1-s}{d(\tilde{v},\tilde{v}')^{n+2s}} \1_{\{ d(\tilde{v},\tilde{v}') \le 1 \}}(\tilde{v},\tilde{v}') \d \tilde{v} \d \tilde{v}' \\
&\ge c (1-s)|v_0|^{-1-2s-\gamma} \int_{B_2}\bigg( \int_{B_3} (g(v) - g(v'))^2 [(1+|v_0 + \tau_0 (v)|^2)(1+|v_0 + \tau_0 (v')|^2)]^{\frac{\gamma+2s+1}{4}} \times \\
& \qquad\qquad\qquad\qquad\qquad \times d(v_0 + \tau_0 (v), v_0 + \tau_0 (v'))^{-n-2s} \1_{\{ d(v_0 + \tau_0 (v), v_0 + \tau_0 (v')) \le 1 \}}( v, v') \d v\bigg) \d v'.
\end{align*}
Next, we make the observation $|\tau_0^{-1}((v_0 + \tau_0 (v)) - (v_0 + \tau_0 (v')))| = |v-v'|$, which implies by \cite[Lemma A.1]{ImSi22} (with $v, v'\in B_3$) that for some universal constant $c_0 \in (0, 1)$:
\begin{align*}
c_0 |v-v'| \le d(v_0 + \tau_0 (v) , v_0 + \tau_0 (v')) \le c_0^{-1} |v-v'|.
\end{align*}
Hence, we deduce
\begin{align*}
\int_{B_2} & \int_{\R^n} (g(v) - g(v'))^2 \tilde{K}_f(v,v') \d v \d v' \\
& \ge c (1-s) |v_0|^{-1-2s-\gamma}  \int_{B_2}\int_{B_{c_0}(v')} \frac{(g(v) - g(v'))^2}{|v-v'|^{n+2s}} [(1+|v_0 + \tau_0 (v)|^2)(1+|v_0 + \tau_0 (v')|^2)]^{\frac{\gamma+2s+1}{4}} \d v \d v'.
\end{align*}

Moreover, we observe that since $|v_0| \ge 2$, for any $v \in B_{2+ c_0}$  it holds that 
\begin{align}
\label{eq:v0-aux-est}
|v_0 + \tau_0 (v)| \ge c|v_0|.
\end{align}
Indeed, by the definition of $\tau_0(v)$, writing $v = a v_0 + w$ for some $w \perp v_0$ and $|a| \le (2+c_0)|v_0|^{-1} \le 3 |v_0|^{-1}$,
\begin{align*}
|v_0 + \tau_0 (v)|^2  = \left| v_0 \left( 1 + \frac{a}{|v_0|} \right) \right|^2 + |w|^2  \ge |v_0|^2 \left( 1 - \frac{3}{|v_0|^2} \right)^2 \ge \left(\frac{|v_0|}{4}\right)^2.
\end{align*}

By \eqref{eq:v0-aux-est}, we have for any $v,v' \in B_{2 + c_0}$
\begin{align*}
[(1+|v_0 + \tau_0 (v)|^2)(1+|v_0 + \tau_0 (v')|^2)]^{\frac{\gamma+2s+1}{4}}\ge c |v_0|^{1+2s+\gamma} .
\end{align*}
Altogether, we have shown that
\begin{align*}
\int_{B_2} & \int_{\R^n} (g(v) - g(v'))^2 \tilde{K}_f(v,v') \d v \d v'  \ge c  (1-s) \int_{B_2}\int_{\R^n} \frac{(g(v) - g(v'))^2}{|v-v'|^{n+2s}}   \1_{\{ |v-v'|\le c_0 \}}( v, v')  \d v \d v'.
\end{align*}
From here, the desired result follows by the same computation as in case $|v_0| \le 2$. The proof is complete.
\end{proof}

\section{Proof of the main result}
\label{sec:main}

In this section we give the proofs of our main results \autoref{thm:entropy-finite} and \autoref{cor:smoothness}.

\subsection{Global H\"older regularity estimates}

First, we establish a global weighted H\"older regularity estimate for solutions to the Boltzmann equation (see \autoref{lemma:Holder-estimate}). The proof goes by application of the H\"older regularity estimate for nonlocal kinetic equations from \cite{ImSi20b} to the Boltzmann equation. The results from the previous sections (see \autoref{thm:change-of-variables-nondegeneracy} and \autoref{thm:change-of-variables-coercivity}) guarantee the applicability of their result in our setting.

The main theorem in \cite[see Theorem 1.5]{ImSi20b} on H\"older regularity for solutions to nonlocal kinetic equations (see also \cite[Theorem 4.2]{ImSi22}) reads as follows.

\begin{proposition}[\cite{ImSi20b}]
\label{prop:Holder-estimate-kinetic}
Let $f \in L^{\infty}((-1,0] \times B_1 \times \R^n)$ be a weak solution to
\begin{align*}
\partial_t f + v \cdot \nabla_x f = \mathcal{L}_K f + h ~~ \text{ in } Q_1
\end{align*}
for some $h \in L^{\infty}(Q_1)$. Assume that $K$ is nonnegative in $(-1,0] \times B_1 \times B_2 \times \R^n$ and that the following hold true for some $0 < \lambda \le \Lambda$, $s_0 \in (0,1)$, and $s \in [s_0,1)$:
\begin{itemize}
\item[(i)] (Upper bound) For any $r > 0$ and any $v \in B_2$:
\begin{align*}
\int_{\R^n \setminus B_r} K(v,v+h) \d h + \int_{\R^n \setminus B_r} K(v+h,v) \d h \le \Lambda r^{-2s}.
\end{align*}
\item[(ii)] \label{it:nondeg}(Nondegeneracy) For any $r > 0$ and $v \in B_2$
\begin{align*}
\inf_{e \in \mathbb{S}^{n-1}} \int_{B_r} K(v,v+h) (h \cdot e)_+^2 \d h \ge \lambda r^{2-2s}> 0 ~~ \text{ if } s < \frac{1}{2}.
\end{align*}
\item[(iii)] \label{it:coerc}(Coercivity) For any $g$ supported in $B_2$:
\begin{align*}
\int_{B_2} \int_{\R^n} (g(v') - g(v))^2 K(v,v') \d v \d v' \ge \lambda [g]_{H^s(\R^n)}^2 - \Lambda \Vert g \Vert_{L^2(\R^n)}^2.
\end{align*}
\item[(iv)] (Cancellation condition) For any $r \in (0,1)$ and $v \in B_2$:
\begin{align*}
\left| \int_{B_r} (K(v,v+h) - K(v+h,v)) \d h \right| &\le \Lambda r^{-2s}, \\
\left| \int_{B_r} (K(v,v+h) - K(v+h,v)) h \d h  \right| &\le \Lambda (1 + r^{1-2s}) ~~ \text{ if } s \ge \frac{1}{2}.
\end{align*}
\end{itemize}
Then, $f$ is H\"older continuous in $Q_r$ and for any $r \in (0,1 )$ and we have
\begin{align*}
[ f ]_{C^{\alpha}_{\ell}(Q_{r/2})} \le C r^{-\alpha} \left( \Vert f \Vert_{L^{\infty}((-r^{2s},0] \times B_{r^{1+2s}} \times \R^n)} + r^{2s} \Vert h \Vert_{L^{\infty}(Q_{r})} \right)
\end{align*}
for some $C > 0$ and $\alpha \in (0,1)$ depending only on $n, s_0, \Lambda,\lambda$.
\end{proposition}

We refer to \cite[Definition 5.7]{ImSi20b} for a definition of weak solutions, and observe that we will apply \autoref{prop:Holder-estimate-kinetic} only for strong solutions to the Boltzmann equation as in \autoref{def:solution}, which are always weak solutions.

\begin{proof}
The result is proved in \cite[see Theorem 1.5]{ImSi20b} (see also \cite[Theorem 4.2]{ImSi22}) for $r = 1$, however with condition (iii) replaced by the assumption that for any $g$ supported in $B_2$ it holds
\begin{align}
\label{eq:coercivity-nonsym}
\int_{B_2} \int_{\R^n} (g(v') - g(v))g(v') K(v,v') \d v \d v' \ge \lambda [g]_{H^s(\R^n)}^2 - \Lambda \Vert g \Vert_{L^2(\R^n)}^2.
\end{align}
Note that under the first cancellation condition in assumption (iv), \eqref{eq:coercivity-nonsym} is equivalent to (iii), as was mentioned in \cite[Proof of Theorem 5.2]{ImSi22}. The result for general $r$ follows immediately by scaling. The proof in \cite{ImSi20b} is robust as $s \to 1$, as was pointed out in \cite[Section 1.2.2]{ImSi22}.
\end{proof}

In the previous sections (see \autoref{thm:change-of-variables-nondegeneracy} and \autoref{thm:change-of-variables-coercivity}), we have seen that the Boltzmann kernel $\tilde K_f$ is still nondegenerate and coercive in our setting. In particular, it satisfies (ii)  and   (iii).

The following lemma was proved in \cite[Theorem 5.2]{ImSi22} and verifies the assumptions (i) and (iv) for the transformed Boltzmann kernel $\tilde{K}_f$ under the macroscopic assumptions \eqref{eq:mass} and \eqref{eq:energy}. It becomes immediately apparent from the proof, that the result is robust as $s \to 1$, and that it remains true for $\gamma + 2s \in [0,q]$ under the assumption \eqref{eq:moment} for $q \ge 2$.

\begin{lemma}[\cite{ImSi22}]
\label{lemma:change-of-variables-1}
Let $q \ge 2$, $s_0 \in (0,1)$, $s \in [s_0,1)$. Let $\gamma \ge 0$ and $\gamma + 2s \in [0,q]$.  Assume that $f$ is nonnegative and satisfies \eqref{eq:mass} and \eqref{eq:moment} for $q \ge 2$.
Then, the kernel $\tilde{K}_f$ given by \eqref{eq:Ktildedef} and \eqref{s-to-1-kernel} with $v_0\in \R^n$  satisfies (i) and (iv) in \autoref{prop:Holder-estimate-kinetic} uniformly in $v_0$, with constants depending only on $n,s_0,m_0,M_0,M_q,\gamma$.
\end{lemma}

By combination of \autoref{lemma:change-of-variables-1}, \autoref{thm:change-of-variables-nondegeneracy}, and \autoref{thm:change-of-variables-coercivity} we are able to apply the previous H\"older estimate, \autoref{prop:Holder-estimate-kinetic}, to the Boltzmann equation in any bounded domain, to obtain the ellipticity conditions in \autoref{prop:properties_kinetic_kernels} uniform as $s\uparrow 1$: 
\begin{lemma}
\label{lem:properties_kinetic_kernels}
Let $s_0\in (0, 1)$, and let $s\in [s_0, 1)$. Let $f$ be nonnegative and satisfying \eqref{eq:mass}, \eqref{eq:temperatureB}, and \eqref{eq:moment} for some $ q > 2$. Then, the Boltzmann kernel $K = \tilde K_f$ given by \eqref{eq:Ktildedef} and \eqref{s-to-1-kernel} with $v_0\in \R^n$ satisfies (i)-(ii)-(iii)-(iv) from \autoref{prop:properties_kinetic_kernels} uniformly in $v_0$, for some constants $\lambda$ and $\Lambda$ depending only on $n$, $s_0$, $\gamma$, $m_0$, $M_0$, $p_0$, $M_q$,and $q$. 
\end{lemma}

We directly prove \autoref{lem:properties_kinetic_kernels}, which in turn implies \autoref{prop:properties_kinetic_kernels} as well. 

\begin{proof}[Proof of \autoref{prop:properties_kinetic_kernels} and \autoref{lem:properties_kinetic_kernels}]
Follows from \autoref{lemma:change-of-variables-1}, \autoref{thm:change-of-variables-nondegeneracy}, and \autoref{thm:change-of-variables-coercivity}. 
\end{proof}

 In order to obtain a global H\"older estimate, we make use of the changes of variables from Section~\ref{sec:prelim}.

Before we apply \autoref{prop:Holder-estimate-kinetic}, we need the following auxiliary lemma. This lemma was already proved in \cite[Lemma 6.3]{ImSi22} for the range $p > n + \gamma + 2s$. We need the result for small values of $p$ as well.

\begin{lemma}
\label{lemma:improved-Lemma-6.3}
Let $q \ge 2$, $\gamma > -n$, $s\in (0, 1)$, $\gamma + 2s \in [0,q]$, and $v_0\in \R^n$. Assume that $f$ is nonnegative and satisfies \eqref{eq:mass} and \eqref{eq:moment} with $q \ge 2$. Let $f \in C^0_{\ell,p}$ for some $p \in [0,n-1)\cup (n+\gamma+2s, +\infty)$. Then, for any $v \in B_1(v_0)$:
\begin{align*}
\int_{M} f(v+h) K_f(v,v+h) \d h \le C (1 + |v_0|)^{-p + \gamma} \Vert f \Vert_{C^0_{\ell,p}((0,T) \times \R^n \times \R^n)},
\end{align*}
where we denote $M = \{ h \in \R^n : |v+h| < |v_0|/8, |h| > 1/2 + |v_0|/8 \}$,  and $C > 0$ depends only on $n,p,M_0,M_q,q$.
\end{lemma}

\begin{proof}
The case $p > n+\gamma+2s$ corresponds to \cite[Lemma 6.3]{ImSi22} with $g = f$. Let us therefore assume $p\in [0, n-1)$. Using \eqref{eq:K-bounds}, as well as the transformation \eqref{eq:weighted-trafo2}, we obtain
\begin{align*}
\int_{M} f(v+h) K_f(v,v+h) \d h &\asymp (1-s) \int_{M} f(v+h) |h|^{-n-2s} \left( \int_{\R^n \cap \{w \perp h\} } f(v+w) |w|^{\gamma + 2s + 1} \d w \right) \d h \\
&= (1-s) \int_{\R^n} f(v+w) |w|^{\gamma + 2s} \left(\int_{M \cap \{ w \perp h \}} f(v+h) |h|^{-n-2s+1} \d h \right) \d w.
\end{align*}
Next, we compute for the inner integral, given any $w \in \R^n$:
\begin{align*}
&\hspace{-1cm}\left(\int_{M \cap \{ w \perp h \}} f(v+h) |h|^{-n-2s+1} \d h \right) \\
&\le \Vert f \Vert_{C^0_{\ell,p}((0,T) \times \R^n \times \R^n)} \left(\int_{M \cap \{ w \perp h \}} |h|^{-(n-1)-2s} (1+ |v+h|)^{-p} \d h\right)\\
&\le C (1 + |v_0|)^{-(n-1)-2s} \left( \int_{\{ |v + h| < |v_0|/8 \} \cap \{ w \perp h \} }  (1 + |v+h|)^{-p} \d h\right) \Vert f \Vert_{C^0_{\ell,p}((0,T) \times \R^n \times \R^n)} \\
&\le C (1 + |v_0|)^{-p -2s} \Vert f \Vert_{C^0_{\ell,p}((0,T) \times \R^n \times \R^n)}.
\end{align*}
Moreover, for the outer integral we have
\begin{align*}
\int_{\R^n} f(v+w) |w|^{\gamma + 2s} \d w &\le C\int_{\R^n} f(w) \left(|w|^{\gamma + 2s} + |v|^{\gamma + 2s} \right) \d w \\
&\le C \left(M_q + (1 + |v_0|)^{\gamma + 2s} M_0 \right),
\end{align*}
where we used that $0 \le \gamma + 2s \le q$ and $|v| \le 1 + |v_0|$.
Therefore,
\begin{align*}
\int_{M} f(v+h) K_f(v,v+h) \d h &\le C (1 + |v_0|)^{-p -2s} \Vert f \Vert_{C^0_{\ell,p}((0,T) \times \R^n \times \R^n)} \left(\int_{\R^n} f(v+w) |w|^{\gamma + 2s} \d w \right) \\
&\le  C (1 + |v_0|)^{-p + \gamma} \Vert f \Vert_{C^0_{\ell,p}((0,T) \times \R^n \times \R^n)},
\end{align*}
as desired.
\end{proof}

Altogether, we obtain a global H\"older estimate. 

\begin{lemma}
\label{lemma:Holder-estimate}
Let $q > 2$, $s_0 \in (0,1)$, $s \in [s_0,1)$. Let $\gamma \ge 0$ and $\gamma + 2s \in [0,q]$ and $T > 0$. Let $f$ be a solution to the Boltzmann equation in $(0, T)\times \R^n\times \R^n$ (see \autoref{def:solution}) satisfying \eqref{eq:mass}, \eqref{eq:temperatureB}, and \eqref{eq:moment} with $q > 2$. Then, there exists $\alpha_0 > 0$ depending only on $n,s_0,m_0,M_0,p_0$, and $ M_q$, such that for all $\alpha\in (0, \alpha_0)$ and $p\in (\alpha, n-1)\cup (n+2s+\gamma, +\infty)$ the following holds: 

If $f\in C^0_{\ell, p}((0, T)\times \R^n\times \R^n)$ then  $f\in C^\alpha_{\ell, p-\alpha}((\tau, T)\times \R^n\times \R^n)$ for any $\tau\in (0, T)$, and the following estimate holds for all  $0 \le \tau_1 < \tau_2 < T$ with $|\tau_2-\tau_1| \le 1$,

\begin{align*}
\Vert f \Vert_{C_{\ell,p-\alpha}^{\alpha}((\tau_2,T)\times \R^n \times \R^n)} \le C (\tau_2 - \tau_1)^{-\frac{\alpha}{2s}} \Vert f \Vert_{C_{\ell,p}^0((\tau_1,T) \times \R^n \times \R^n)},
\end{align*}
where $C > 0$ depends only on $n,s_0,p, m_0,M_0,p_0,q, M_q$.
\end{lemma}

\begin{proof}
Note that the claim follows once we show that for any $z_0,z \in (\tau_2,T) \times \R^n \times \R^n$ it holds
\begin{align}
\label{eq:claim-Holder}
|f(z_0) - f(z)| \le C (\tau_2 - \tau_1)^{-\frac{\alpha}{2s}} d_{\ell}(z_0,z)^{\alpha} (1 + |v_0|)^{-p+\alpha} \Vert f \Vert_{C^0_{\ell,p} ((\tau_1,T) \times \R^n \times \R^n)}.
\end{align}

Let us set $r := (\tau_2 - \tau_1)^{\frac{1}{2s}}$. Then, for any $z_0 \in (\tau_2,T) \times \R^n \times \R^n$, we have $Q_{r}(z_0) \subset (\tau_1,T) \times \R^n \times \R^n$.\\
We fix $z_0 = (t_0,x_0,v_0) \in (\tau_2,T) \times \R^n \times \R^n$, and divide the proof into two steps. 

{\bf Step 1.} First, we consider the case $|v_0| \le 2$. We apply the change of variables $\mathcal{T}_0$ to $f$, set $\tilde{f}(t,x,v) = f(\tilde{t},\tilde{x},\tilde{v}) = f(t_0 + t , x_0 + x + t v_0 , v_0 + v)$ and observe that $\tilde{f}$ solves (recall \eqref{eq:Qfg})
\begin{align*}
\partial_t \tilde{f} + v \cdot \nabla_x \tilde{f} = \mathcal{L}_{\tilde{K}_f} \tilde{f} + \tilde{h} ~~ \text{ in } Q_r,
\end{align*}
where
\begin{align*}
\tilde{h}(t,x,v) = c_b (\tilde{f} \ast_v |\cdot|^{\gamma}) \tilde{f}(t,x,v).
\end{align*}
By the mass and moment bound, and since $\gamma \in [0, q]$ (alternatively, by \cite[Lemma 2.3]{ImSi22}, which works in the exact same way if $q > 2$) we have for some $C > 0$, depending only on $M_0,M_q$, and $q$,
\begin{align*}
\Vert \tilde{h} \Vert_{L^{\infty}(Q_r)}  \le C(1-s) \Vert \tilde{f} \Vert_{L^{\infty}(Q_r)} (1 + |v_0|)^{\gamma} = C \Vert f \Vert_{L^{\infty}(Q_r(z_0))} (1 + |v_0|)^{\gamma} \le C \Vert f \Vert_{L^\infty((\tau_1,T) \times \R^n \times \R^n)},
\end{align*}
where we also used that $|v_0| \le 2$. Moreover, by \autoref{lemma:change-of-variables-1}, \autoref{thm:change-of-variables-nondegeneracy}, and \autoref{thm:change-of-variables-coercivity}, the kernel $\tilde{K}_f$ satisfies the assumptions (i), (ii), (iii), and (iv) of \autoref{prop:Holder-estimate-kinetic} (see \autoref{lem:properties_kinetic_kernels}). Thus, an application of  \autoref{prop:Holder-estimate-kinetic} to $\tilde{K}_f$ yields that for any $z_1,z_2 \in Q_{r/2}$:
\begin{align*}
|\tilde{f}(z_1) - \tilde{f}(z_2)| \le C r^{-\alpha} (\Vert \tilde{f} \Vert_{L^{\infty}((-r^{2s},0) \times B_{r^{1+2s}} \times \R^n)} + r^{2s} \Vert \tilde{h} \Vert_{L^{\infty}(Q_r)} ) d_{\ell}(z_1,z_2)^{\alpha},
\end{align*}
where $d_\ell$ denotes the kinetic distance (recall \autoref{def:kin_dist}). Undoing the change of variables, and choosing $z_1 = 0$ implies that for any $\tilde{z}_2 \in Q_{r/2}(z_0)$:
\begin{align*}
|f(z_0) - f(\tilde{z}_2)| &= |\tilde{f}(0) - \tilde{f}(z_2)| \le C r^{-\alpha} (\Vert \tilde{f} \Vert_{L^{\infty}((-r^{2s},0) \times \R^n \times \R^n)} + r^{2s} \Vert f \Vert_{L^\infty((\tau_1,T) \times \R^n \times \R^n)} ) d_{\ell}(0,z_2)^{\alpha} \\
&\le C r^{-\alpha} \Vert f \Vert_{L^\infty((\tau_1,T) \times \R^n \times \R^n)} d_{\ell}(z_0,\tilde{z}_2)^{\alpha},
\end{align*}
where we also used that $d_{\ell}(0,z_2) = d_{\ell}(z_0,\tilde{z}_2)$. This immediately implies \eqref{eq:claim-Holder} for $z_0 \in (\tau_2,T) \times \R^n \times B_2$.

{\bf Step 2.} Let us now consider $z_0 \in (\tau_2,T) \times \R^n \times \R^n$ with $|v_0| > 2$. Let $\phi \in C^{\infty}(\R^n)$ be a cut-off function that is supported in $B_{|v_0|/8}$, satisfies $0 \le \phi \le 1$, and $\phi \equiv 1$ in $B_{|v_0|/9}$. In particular, note that $\phi$ vanishes in $E_1(v_0)$. Then, we define $g(t,x,v) = (1-\phi(v)) f(t,x,v)$ and observe that $g$ solves
\begin{align*}
\partial_t g + v \cdot \nabla_x g = \mathcal{L}_{K_f}g + h_1 + h_2 ~~ \text{ in } (0,T) \times \R^n \times E_r(v_0).
\end{align*}
Here,
\begin{align*}
h_1(t,x,v) = \int_{\R^n} \phi(v+h) f(t,x,v+h) K_f(v,v+h) \d h , \qquad h_2(t,x,v) = c_b (f \ast_v |\cdot|^{\gamma}) f(t,x,v),
\end{align*}
As before, by the mass and moment bound,
\begin{align}
\label{eq:h2-est}
\Vert h_2 \Vert_{L^{\infty}(\mathcal{E}_r(z_0))}  \le C \Vert f \Vert_{L^{\infty}(\mathcal{E}_r(z_0))} (1 + |v_0|)^{\gamma} \le  C(1 + |v_0|)^{-p + \gamma}  \Vert f \Vert_{C^0_{\ell,p}((\tau_1,T) \times \R^n \times \R^n)}.
\end{align}
Moreover, for $h_1$ we observe that by construction of $\phi$, and for $v\in B_1(v_0)$, the domain of integration  is restricted to $M = \{ h \in \R^n : |v+h| < |v_0|/8, |h| > 1/2 + |v_0|/8 \}$. Hence, we obtain by \autoref{lemma:improved-Lemma-6.3}, and using that $\mathcal{E}_r(z_0) \subset \mathcal{E}_1(z_0) \subset Q_1(z_0) \subset (\tau_1 , T) \times \R^n \times \R^n$:
\begin{align}
\label{eq:h1-est}
\Vert h_1 \Vert_{L^{\infty}(\mathcal{E}_r(z_0))} \le \sup_{v\in B_1(v_0)}\int_{M} f(t, x, v+h) K_f(v,v+h) \d h \le C (1 + |v_0|)^{-p+\gamma} \Vert f \Vert_{C^0_{\ell,p}((\tau_1,T) \times \R^n \times \R^n)}.
\end{align}

Then, we apply the change of variables $\mathcal{T}_0$ to $g$, set $\tilde{g}(t,x,v) = g(\tilde{t},\tilde{x},\tilde{v})$ and observe that $\tilde{g}$ solves
\begin{align*}
\partial_t \tilde{g} + v \cdot \nabla_x \tilde{g} = \mathcal{L}_{\tilde{K}_f} \tilde{g} + \tilde{h} ~~ \text{ in } Q_r,
\end{align*}
where $\tilde{h}(t,x,v) = |v_0|^{-\gamma-2s}(h_1(\tilde{t},\tilde{x},\tilde{v}) + h_2(\tilde{t},\tilde{x},\tilde{v}))$. By \autoref{lemma:change-of-variables-1}, \autoref{thm:change-of-variables-nondegeneracy}, and \autoref{thm:change-of-variables-coercivity}, the kernel $\tilde{K}_f$ satisfies the assumptions (i), (ii), (iii), and (iv) of \autoref{prop:Holder-estimate-kinetic}. Therefore, \autoref{prop:Holder-estimate-kinetic} is applicable to $\tilde{K}_f$, and for any $z_1,z_2 \in Q_{r/2}$:
\begin{align*}
|\tilde{g}(z_1) - \tilde{g}(z_2)| \le C r^{-\alpha} (\Vert \tilde{g} \Vert_{L^{\infty}((-r^{2s},0) \times B_{r^{1+2s}} \times \R^n)} + r^{2s} \Vert \tilde{h} \Vert_{L^{\infty}(Q_{r})} ) d_{\ell}(z_1,z_2)^{\alpha}.
\end{align*}
Now, by construction
\begin{align*}
\Vert \tilde{g} \Vert_{L^{\infty}((-r^{2s},0) \times B_{r^{1+2s}} \times \R^n)} \le \Vert \tilde{f} \Vert_{L^{\infty}((-r^{2s},0) \times B_{r^{1+2s}} \times (\R^n \setminus B_{|v_0|/9}))} \le C (1 + |v_0|)^{-p} \Vert f \Vert_{C^0_{\ell,p}((\tau_1,T) \times \R^n \times \R^n)},
\end{align*}
and by \eqref{eq:h2-est} and \eqref{eq:h1-est}, using also that if $z \in Q_r$, then $\tilde{z} \in \cE_r(z_0) \subset (\tau_1,T) \times \R^n \times \R^n$:
\begin{align*}
\Vert \tilde{h} \Vert_{L^{\infty}(Q_r)} \le  |v_0|^{-\gamma-2s} \Vert h_1 \Vert_{L^{\infty}(\cE_r(z_0))} + |v_0|^{-\gamma-2s} \Vert h_2 \Vert_{L^{\infty}(\cE_r(z_0))} \le 
C(1+|v_0|)^{-p-2s}  \Vert f \Vert_{C^0_{\ell,p}((\tau_1,T) \times \R^n \times \R^n)}.
\end{align*}
Altogether, choosing $z_1 = 0$, this implies for any $\tilde{z}_2 \in \mathcal{E}_{r/2}(z_0)$:
\begin{align*}
|f(z_0) - f(\tilde{z}_2)| = |\tilde{f}(0) - \tilde{f}(z_2)| &= |\tilde{g}(0) - \tilde{g}(z_2)| \le C r^{-\alpha} (1+|v_0|)^{-p}  \Vert f \Vert_{C^0_{\ell,p}((\tau_1,T) \times \R^n \times \R^n)} d_{\ell}(0,z_2)^{\alpha} \\
&\le C r^{-\alpha} (1+|v_0|)^{-p+\alpha}  \Vert f \Vert_{C^0_{\ell,p}((\tau_1,T) \times \R^n \times \R^n)} d_{\ell}(z_0,\tilde{z}_2)^{\alpha},
\end{align*}
where we also used that $\tilde{g} = \tilde{f}$ in $Q_{1/2}$ and $d_{\ell}(z_1,z_2) \le (1+|v_0|) d_{\ell}(\tilde{z}_1,\tilde{z}_2)$. Moreover, if $\tilde{z}_2 \not \in \mathcal{E}_{r/2}(z_0)$, we have $d_{\ell}(z_0,\tilde{z}_2) \ge c r (1 + |v_0|)^{-1}$ and therefore it holds
\begin{align*}
& |f(z_0) - f(\tilde{z}_2)| \le 2 \Vert f \Vert_{L^{\infty}((\tau_1,T) \times \R^n \times \R^n)} \\
&\qquad \le C (1 + |v_0|)^{-p} \Vert f \Vert_{C^0_{\ell,p}((\tau_1,T) \times \R^n \times \R^n)} \le C r^{-\alpha} (1+|v_0|)^{-p+\alpha}  \Vert f \Vert_{C^0_{\ell,p}((\tau_1,T) \times \R^n \times \R^n)} d_{\ell}(z_0,\tilde{z}_2)^{\alpha}.
\end{align*}
Hence, we have   verified \eqref{eq:claim-Holder} also  for $z_0 \in (\tau_2,T) \times \R^n \times \R^n$ with $|v_0| > 2$ and the proof is complete.
\end{proof}

Note that in \eqref{eq:h2-est} we heavily used that $\gamma \ge 0$. For $\gamma < 0$, the bound in \eqref{eq:h2-est} would have a nonlinear dependence on $\Vert f \Vert_{C_{\ell,p}^0}$ which is incompatible with our strategy tho prove \autoref{thm:Holder-estimate}.

\subsection{Replacing decay estimates by interpolation}

The estimate in \autoref{lemma:Holder-estimate} is not satisfactory, since we do not have an estimate of $\Vert f \Vert_{C_{\ell,p}^0}$ only in terms of universal constants and the macroscopic bounds. Such decay estimate was established in \cite{Sil16,IMS20}, but the proof heavily relies on the existence of nondegeneracy cones, which in turn relies on the boundedness of the entropy. It seems like the technique in \cite{Sil16,IMS20} cannot easily be generalized so that it works solely under temperature and moment bounds.

We will circumvent proving decay estimates by establishing an  interpolation result, which allows us to estimate the $C^{\alpha}_{\ell,p-\alpha}$ norm from \autoref{lemma:Holder-estimate} by a higher moment of $f$.

We have the following interpolation result.

\begin{lemma}
\label{lemma:interpolation-kinetic}
Let $f \in C^{\alpha}_{\ell,p}((\tau,T) \times \R^n \times \R^n)$ for some $\alpha \in (0,1)$. Then, for any $r > 0$
\begin{align*}
\Vert f \Vert_{C^0_{\ell,p}((\tau,T) \times \R^n \times \R^n)} \le r^{\alpha} \Vert f \Vert_{C^{\alpha}_{\ell,p}((\tau,T) \times \R^n \times \R^n)} + C r^{-n} \Vert f \Vert_{L^{\infty}_{t,x}L^1_{\ell,p}((\tau,T) \times \R^n \times \R^n)},
\end{align*}
where $C > 0$ depends only on $n$. Moreover, it holds for any $\eps \in (0, 1)$:
\begin{align*}
\Vert f \Vert_{C^0_{\ell,p}((\tau,T) \times \R^n \times \R^n)} \le \eps^{\alpha} \Vert f \Vert_{C^{\alpha}_{\ell,p-\alpha}((\tau,T) \times \R^n \times \R^n)} + C \eps^{-n} \Vert f \Vert_{L^{\infty}_{t,x}L^1_{\ell,p+n}((\tau,T) \times \R^n \times \R^n)}.
\end{align*}
\end{lemma}

\begin{proof}
Let $(t, x)\in (\tau, T)\times \R^n$ be fixed. First, we claim that for any $r > 0$ and $v \in \R^n$:
\begin{align*}
|f(t, x, v)|\le r^\alpha\|f(t, x, \cdot)\|_{C^\alpha_v(B_r(v))}+  Cr^{-n} \|f(t, x, \cdot)\|_{L^1(B_r(v))}. 
\end{align*}

To see this, we compute
\begin{align*}
|f(t,x,v)| &\le |f(t,x,v) - (f(t,x,\cdot))_{r,v}| + {f(t,x,\cdot)}_{r,v} \\
&\le \sup_{w \in B_r(v)} |f(t,x,v) - f(t,x,w)| + c r^{-n} \Vert f(t,x,\cdot) \Vert_{L^1(B_r(v))} \\
&\le r^{\alpha} \Vert f(t,x,\cdot) \Vert_{C^{\alpha}(B_{r}(v))} + c r^{-n} \Vert f(t,x,\cdot) \Vert_{L^1(B_{r}(v))},
\end{align*}
where we denoted $(f(t,x,\cdot))_{r,v} = \dashint_{B_r(v)} f(t,x,w) \d w$.

Multiplying both sides of the estimate by $(1+|v|)^p$ 
\[
\begin{split}
(1+|v|)^p|f(t, x, v)|& \le r^\alpha (1+|v|)^p \|f(t, x, \cdot)\|_{C^\alpha_v(B_r(v))}+  Cr^{-n} (1+|v|)^p\|f(t, x, \cdot)\|_{L^1(B_r(v))}
\end{split}
\]
so that taking the supremum over $v$ gives the first inequality. Moreover, fixing $r = \eps (1+|v|)^{-1}<1$,
\[
\begin{split}
(1+|v|)^p|f(t, x, v)|
& \le  \eps^{\alpha} (1+|v|)^{p-\alpha} \Vert f(t, x, \cdot) \Vert_{C^{\alpha}(B_r(v))} + C \eps^{-n} (1+|v|)^{p + n}\Vert f(t, x, \cdot) \Vert_{L^1(B_r(v))},
\end{split}
\]
we obtain, after taking again the supremum over $v$, the second inequality. This concludes the proof. 
\end{proof}

Moreover, we will make use of the following standard iteration lemma, which can be found for instance in \cite[Lemma 6.1]{Giu03}:

\begin{lemma}
\label{lemma:Giusti}
Let $F : [T_1, T_2] \mapsto [0,\infty)$ be bounded. Assume that there are $A \ge 0$ and $\gamma > 0$  such that for every $T_1 \le t_1 < t_2 \le T_2$ it holds
\begin{align*}
F(t_1) \le \frac{1}{2}F(t_2) + A (t_2 - t_1)^{-\gamma}.
\end{align*}
Then there is a constant $c > 0$, depending only on $\gamma$, such that for every $T_1 \le s_1 < s_2 \le T_2$
\begin{align*}
F(s_1) \le c A (s_2 - s_1)^{-\gamma}.
\end{align*}
\end{lemma}

For the sake of completeness, let us give a proof of this lemma.

\begin{proof}
We set $\tau_0 = s_1$ and $\tau_{i+1} = \tau_i + (1 - \sigma) \sigma^i (s_2 - s_1)$ for some $\sigma \in (0,1)$ to be chosen later. Then,
\begin{align*}
F(s_1) = F(\tau_0) \le 2^{-k} F(\tau_k) + A (1 - \sigma)^{-\gamma} (s_2 - s_1)^{-\gamma} \sum_{i = 0}^{k-1} 2^{-i} \sigma^{-i \gamma}.
\end{align*}
Choosing $\sigma = 2^{-\frac{1}{2\gamma}}$ we obtain
\begin{align*}
F(s_1) \le \lim_{k \to \infty} 2^{-k} \big(\sup_j F(\tau_j) \big)  + A (1 - \sigma)^{-\gamma} (s_2 - s_1)^{-\gamma} \sum_{ i = 0}^{\infty} 2^{-\frac{i}{2}} \le c A (s_2 - s_1)^{-\gamma}.
\end{align*}
\end{proof}

By combination of \autoref{lemma:Holder-estimate} and \autoref{lemma:interpolation-kinetic}, we obtain a global H\"older estimate only in terms of macroscopic bounds and universal constants.

\begin{theorem}
\label{thm:Holder-estimate}
Let $q > n$, $s_0 \in (0,1)$, $s \in [s_0,1)$. Let $\gamma \ge 0$, and $\gamma + 2s \in [0,q]$. Let $f$ be a solution to the Boltzmann equation in $(0,T) \times \R^n \times \R^n$ (see \autoref{def:solution}) satisfying \eqref{eq:mass}, \eqref{eq:temperatureB}, and \eqref{eq:moment}. Assume, in addition, that either $q < 2n-1 $ or $q > 2n+\gamma+2s$. \\
If $f \in C^0_{\ell,q-n}((0,T) \times \R^n \times \R^n)$, then $f \in C^{\alpha}_{\ell,p}((\tau,T) \times \R^n \times \R^n)$ for any $\tau \in (0,T)$ and the following estimate holds true 
\begin{align*}
\Vert f \Vert_{C_{\ell,p}^{\alpha}((\tau,T)\times \R^n \times \R^n)} \le C,\qquad p = q-n-\alpha,
\end{align*}
for some $\alpha \in (0,q-n)$ and $C > 0$ depending only on $n,s_0,m_0,M_0,p_0,M_q,q$, and $C$ depending also on $\tau$.
\end{theorem}

\begin{proof}
Let us set $\bar p = q - n$ and fix $\tau \in (0,T)$. 
By application of \autoref{lemma:Holder-estimate} and \autoref{lemma:interpolation-kinetic} we have for any $\max \{0 , \tau - 1 \} \le \tau_1 < \tau_2 \le \tau < T$ and any $\eps \in (0,1)$:
\begin{align*}
\Vert f \Vert_{C_{\ell,\bar p-\alpha}^{\alpha}((\tau_2,T)\times \R^n \times \R^n)} &\le C (\tau_2 - \tau_1)^{-\frac{\alpha}{2s}} \Vert f \Vert_{C_{\ell,\bar p}^0((\tau_1,T) \times \R^n \times \R^n)} \\
&\le C (\tau_2 - \tau_1)^{-\frac{\alpha}{2s}} \eps^{\alpha}  \Vert f \Vert_{C_{\ell,\bar p-\alpha}^{\alpha}((\tau_1,T)\times \R^n \times \R^n)} + (\tau_2 - \tau_1)^{-\frac{\alpha}{2s}} C\eps^{-n} M_{q},
\end{align*}
for any $\alpha \in (0, \min \{\bar p,\alpha_0\})$, where $\alpha_0 > 0$ and $C > 0$ depend only on $n,s_0,m_0,M_0, M_q,p,p_0$. Here we also used that $\Vert f \Vert_{L^{\infty}_{t,x}L^1_{\ell,\bar p+n}((\tau_1,T) \times \R^n \times \R^n)} \le M_{\bar p+n} = M_q$.
Next, let us fix $\alpha \in (0,\min\{p,\alpha_0\})$ and choose $\eps = (\tau_2 - \tau_1)^{\frac{1}{2s}}(2C)^{-\frac{1}{\alpha}}$. Then, we have shown that for any $\max\{0, \tau-1 \} \le \tau_1 < \tau_2 \le \tau < T$:
\begin{align*}
\Vert f \Vert_{C_{\ell,\bar p-\alpha}^{\alpha}((\tau_2,T)\times \R^n \times \R^n)} \le \frac{1}{2}\Vert f \Vert_{C_{\ell,\bar p-\alpha}^{\alpha}((\tau_1,T)\times \R^n \times \R^n)} + C_2 (\tau_2 - \tau_1)^{-\frac{n+\alpha}{2s}}  M_q
\end{align*}
for some $C_2 > 0$, depending only on $C,\alpha$. Let us now denote 
\begin{align*}
F(r) = \Vert f \Vert_{C_{\ell,\bar p-\alpha}^{\alpha}((T-r,T)\times \R^n \times \R^n)}.
\end{align*}
The aforementioned statement reads now as follows: for any $T- \tau \le t_1 < t_2 \le \min \{T , T - (\tau - 1)\}$ it holds
\begin{align*}
F(t_1) \le \frac{1}{2} F(t_2) + C_2 M_q (t_2 - t_1)^{-\frac{n+\alpha}{2s}}.
\end{align*}
Note that $F$ is bounded since $f \in C^{\alpha}_{\ell,\bar p - \alpha}((\tau,T) \times \R^n \times \R^n)$ for any $\tau \in (0,T)$ due to the assumption that $f \in C^{0}_{\ell,\bar p}((0,T) \times \R^n \times \R^n)$ and \autoref{lemma:Holder-estimate}. Thus, we can apply \autoref{lemma:Giusti} to $F$ and deduce that
\begin{align*}
\Vert f \Vert_{C_{\ell,\bar p-\alpha}^{\alpha}((\tau,T)\times \R^n \times \R^n)} &= F(T-\tau) \\
&\le  C_3 (\min\{T, T-(\tau-1)\} - (T-\tau))^{-\frac{n+\alpha}{2s}} M_q = C_3 \min\{\tau,1\}^{-\frac{n+\alpha}{2s}} M_q
\end{align*}
for some $C_3 > 0$, depending only on $C_2,n,s_0,\alpha$.
\end{proof}

\subsection{Smoothness of solutions}

As a consequence of \autoref{thm:Holder-estimate}, we deduce that the entropy is finite.

\begin{proof}[Proof of \autoref{thm:entropy-finite} and \autoref{thm:grazing}]
The result follows directly from \autoref{thm:Holder-estimate}. Observe that, given $\gamma$ and $s$ fixed, we can obtain any $p$ very large by fixing $q = n+p+1$  in \autoref{thm:Holder-estimate}, for example.
\end{proof}

As a consequence, we deduce \autoref{cor:smoothness}.

\begin{proof}[Proof of \autoref{cor:smoothness}]
Due to \autoref{thm:entropy-finite}, we can follow the regularity program of Imbert--Silvestre and obtain the $C^{\infty}$ regularity of solutions to the Boltzmann equation in the same way as in \cite{ImSi22}.
\end{proof}

\section{Regularity for the Landau equation}
\label{sec:Landau}

The goal of this section is to explain how to establish \autoref{cor:Landau} using the techniques developed in this article.

First, we recall the kinetic cylinders $Q_r(z_0)$ and the change of variables $\tau_0$ and $\mathcal{T}_0$ from Section \ref{sec:aux}. For the Landau equation, we define them in the exact same way, setting $s = 1$. Note that when $f$ solves the Landau equation in $\mathcal{E}_1(z_0)$, \eqref{eq:Landau}, then $\tilde{f}$ solves
\begin{align}
\label{eq:Landau-transformed}
\partial_t \tilde{f} + v \cdot \nabla_x \tilde{f} = \nabla_v [\tilde{A} \nabla_v \tilde{f}] + \tilde{b} \cdot \nabla_v \tilde{f} + \tilde{c} \tilde{f}  ~~ \text{ in } Q_1,
\end{align}

where (recall \eqref{eq:a_Landau}-\eqref{eq:b_Landau}-\eqref{eq:c_Landau})
\begin{align*}
\tilde{A}(t,x,v) e &= \begin{cases}
|v_0|^{-\gamma-2} \tau_0^{-1} (A(\tilde{t},\tilde{x},\tilde{v}) \tau_0^{-1}e), ~~ \text{ if } |v_0| \ge 2,\\
A(\tilde{t},\tilde{x},\tilde{v})e ~\quad~\quad\qquad\qquad\qquad \text{ if } |v_0| < 2,
\end{cases}\quad\text{for any}\quad e\in \R^n,\\
\tilde{b}(t,x,v) &= \begin{cases}
|v_0|^{-\gamma-2} \tau_0^{-1} b(\tilde{t},\tilde{x},\tilde{v}), ~~ \text{ if } |v_0| \ge 2,\\
b(\tilde{t},\tilde{x},\tilde{v}) ~~~~\quad\qquad\qquad \text{ if } |v_0| < 2,
\end{cases}\\
\tilde{c}(t,x,v) &= \begin{cases}
|v_0|^{-\gamma-2} c(\tilde{t},\tilde{x},\tilde{v}), ~~ \text{ if } |v_0| \ge 2,\\
c(\tilde{t},\tilde{x},\tilde{v}) ~~~~~\quad\qquad  \text{ if } |v_0| < 2,
\end{cases}
\end{align*}


In order to prove \autoref{cor:Landau}, we first need to establish uniform ellipticity of the transformed matrix $\tilde{A}$ in $B_2$ (in analogy to \autoref{prop:properties_kinetic_kernels}). Moreover, we require suitable upper bounds for the lower order terms $\tilde{b}$ and $\tilde{c}$.

For the lower bound in the uniform ellipticity of $\tilde{a}$, we use the pressure lower bound on $f$ and proceed in the same way as in the proof of \autoref{thm:change-of-variables-nondegeneracy}: 

\begin{lemma}
\label{lemma:Landau-ellipticity}
Let $\gamma > -n$.  Assume that $f$ is nonnegative and satisfies \eqref{eq:mass}, \eqref{eq:temperatureB}, and \eqref{eq:moment} for some $q > 2$. Then, $\tilde{A}$ with $v_0\in \R^n$ satisfies 
\begin{align*}
e \cdot \tilde{A}(v) e  \ge \lambda   \qquad\text{for all}\quad v \in B_2, ~~ e \in  \mathbb{S}^{n-1}
\end{align*}
uniformly in $v_0$, with $\lambda > 0$ depending only on $n,m_0,M_0,p_0,M_q,q,$ and $\gamma$.
\end{lemma}

\begin{proof}
First, we explain how to estimate $A$ (from \eqref{eq:a_Landau}) without applying the change of variables. We claim that 
there exists $\lambda > 0$ such that
\begin{align}
\label{eq:unif-ell-Landau-1}
e \cdot A(v) e  \ge \lambda (1 + |v|)^{\gamma}   ~~ \forall v \in \R^n, ~~ e \in \mathbb{S}^{n-1}. 
\end{align}

To see this, note that by \autoref{prop:subtract-tube} we have that 
\begin{align*}
\int_{B_R(\overline{v}) \setminus L_{\delta} } f(w) \d w \ge c
\end{align*}
for some $R > 0$, $\delta,c \ge 0$, depending only on $m_0,M_0,p_0,M_q,q$, where we denote by $L_{\delta}$ the tube of radius $\delta$ around $v + \R e$. Then, we have for any $v \in \R^n$ and $e \in \mathbb{S}^{n-1}$
\begin{align*}
e \cdot A(v) e  =  a_{n, \gamma } \int_{\R^n} G(w,e) |w|^{\gamma + 2} f(v-w) \d w = a_{n,\gamma} \int_{\R^n} G(w-v,e) |w-v|^{\gamma + 2} f(w) \d w,
\end{align*}
where 
\begin{align*}
G(w-v,e) :=  1 - \frac{[(w-v) \cdot e]^2}{|w-v|^2} =  1 - \cos^2(w-v,e) =  \sin^2(w-v,e).
\end{align*}
Hence, by following exactly the same arguments as in the proof of \autoref{thm:Boltzmann-kernel-nondegenerate}, we obtain that \eqref{eq:Gcomp}, \eqref{eq:G-estimate-case-1}, and \eqref{eq:eq-case-1_Br} also hold true in our setup, and thus we deduce \eqref{eq:nondegenerate-goal} with $s = 1$. Hence, the proof of \eqref{eq:unif-ell-Landau-1} is complete.

We are now in a position to prove the desired result. First, note that we are done  when $|v_0| \le 2$ by \eqref{eq:unif-ell-Landau-1}. When $|v_0| > 2$,   since 
\begin{align*}
\tau_0^{-1} \left\{ \left( I - \frac{w}{|w|} \otimes \frac{w}{|w|} \right) \tau_0^{-1}(e) \right\} \cdot e = \left( I - \frac{w}{|w|} \otimes \frac{w}{|w|} \right) \tau_0^{-1}(e) \cdot  \tau_0^{-1}(e), 
\end{align*}
we have that it holds
\begin{align*}
e \cdot \tilde{A}(v) e  = |v_0|^{-\gamma-2} \int_{\R^n} G(w-\tilde{v},\tau_0^{-1}(e)) |w-\tilde{v}|^{\gamma + 2} f(w) \d w.
\end{align*}
In case $2 \le |v_0| \le 10(R + |\overline{v}| + 1)$, we can apply the proof of \eqref{eq:unif-ell-Landau-1} with $v := \tilde{v}$ and $e := \tau_0^{-1}(e)$ and obtain
\begin{align*}
e \cdot \tilde{A}(v) e\ge \lambda (1 + |\tilde{v} - \overline{v}|)^{\gamma}  \ge c  ,
\end{align*}
where we used that $|\tilde{v}| + |\overline{v}| \le C$ when $|v_0| \le 10(R + |\overline{v}| + 1)$. \\
In case $|v_0| \ge 10(R + |\overline{v}| + 1)$, we recall also \eqref{eq:t0-norm-compute}, and deduce that
\begin{align*}
G(w-\tilde{v},\tau_0^{-1}(e)) &= |\tau_0^{-1}(e)|^2  - \frac{[(w - \tilde{v}) \cdot \tau_0^{-1}(e)]^2}{|w - \tilde{v}|^2} = |\tau_0^{-1}(e)|^2 \sin(\tilde{w} - \tilde{v} , \tau_0^{-1}(e))\\
& = [1 + (|v_0|^2 - 1)\cos^2(v_0,e)] \sin^2(w - \tilde{v} , \tau_0^{-1}(e)).
\end{align*}
Hence, in this case we can apply exactly the same arguments as in Step 3 of the proof of \autoref{thm:change-of-variables-nondegeneracy}. In particular, we can use \eqref{eq:trafo-nondegeneracy-help-0}, which immediately implies the desired result.
\end{proof}

For the remaining properties, we recall that the aforementioned change of variables has already been used in \cite{CSS18}, \cite{HeSn20} (in case $\gamma \le 0$). It was already shown in \cite[Lemma 2.4]{Sne20} that the the same computations carry over to hard potentials (in case $\gamma \in (0,1]$). For the convenience of the reader, we redo the computations in our setting in the following lemma.

\begin{lemma}
\label{lemma:Landau-upper-ellipticity}
Let $q \ge 2$. Let $\gamma \ge 0$ and $\gamma + 2 \in [0,q]$. Assume that $f$ is nonnegative and satisfies \eqref{eq:mass}, and \eqref{eq:moment}  with $q \ge 2$. Then, $\tilde{A}, \tilde{b}, \tilde{c}$ with $v_0\in \R^n$ satisfy
\begin{align*}
\sup_{e \in \mathbb{S}^{n-1}} e \cdot \tilde{A}(v) e \le \Lambda  , \qquad
|\tilde{b}(v)| \le \Lambda, \qquad |\tilde{c}(v)| \le \Lambda (1 + |v_0|)^{-2},\qquad\text{for all}\quad v\in B_2,
\end{align*}
with $\Lambda > 0$ depending only on $n,M_0,M_q,$ and $q$.
\end{lemma}

\begin{proof}
Following the proof of \cite[Lemma 2.1]{CSS18} it becomes apparent that also for $\gamma \ge 0$ it holds
\begin{align}
\label{eq:a-upper}
e \cdot A(v) e \le C
\begin{cases}
(1 + |v|)^{\gamma + 2}, ~~ e \in \mathbb{S}^{n-1},\\
(1 + |v|)^{\gamma} , ~~~~~ v \parallel e\in \mathbb{S}^{n-1},
\end{cases}
\end{align}
where $C > 0$ depends on $n,M_0,M_q,q$. The modifications to the proof of  \cite[Lemma 2.1]{CSS18} are obvious in the first case. If $e \parallel v$, we compute
\begin{align*}
e \cdot A(v) e &=   \int_{\R^d} |w|^2 \sin^2(v,w) |v-w|^{\gamma} f(w) \d w \\
&\le c    \int_{\R^d} |w|^{\gamma + 2} f(w) \d w + c  |v|^{\gamma} \int_{\R^d} |w|^2 f(w) \d w \le c (1 + |v|)^{\gamma}  . 
\end{align*}
The first identity is proved in \cite[Lemma 2.1]{CSS18}.
From \eqref{eq:a-upper}, we deduce the desired estimate for $\tilde{A}$ by following the corresponding arguments in the proof of \cite[Lemma 4.1]{CSS18}.

To prove the estimates for $b$ and $c$, we observe that 
\begin{align*}
|b(v)| \le C (1 + |v|)^{1+\gamma}, \qquad |c(v)| \le C (1 + |v|)^{\gamma},
\end{align*}
where $C > 0$ depends on $n,M_0,M_q,q$.
The proof of the estimates for $b$ is the same as in \cite[Lemma 2.3]{CSS18} in case $\gamma \in [-1,0]$ and the proof for $c$ goes in the same way, replacing $1 + \gamma$ by $\gamma$. From here, the estimates for $\tilde{b}$ and $\tilde{c}$ follow from the fact that $|\tilde{v}| \le c (1 + |v_0|)$ and $\Vert \tau_0^{-1} \Vert \le (1 + |v_0|)$.
\end{proof}

Having at hand \autoref{lemma:Landau-ellipticity} and \autoref{lemma:Landau-upper-ellipticity}, we are now in a position to prove a global H\"older estimate for solutions to the Landau equation. This results and its proof are in analogy to \autoref{lemma:Holder-estimate}, using the $C^{\alpha}$ estimate from \cite{GIMV19}.

\begin{lemma}
\label{lemma:Holder-estimate-Landau}
Let $q > 2$. Let $\gamma \ge 0$ and $\gamma + 2 \in [0,q]$ and $T > 0$. Let $f$ be a weak solution to the Landau equation in $(0, T)\times \R^n\times \R^n$ satisfying \eqref{eq:mass}, \eqref{eq:temperatureB}, and \eqref{eq:moment} with $q > 2$. Then, there exists $\alpha_0 > 0$ depending only on $n,m_0,M_0,p_0$, and $ M_q$, such that for all $\alpha\in (0, \alpha_0)$ and $p\in (\alpha,+\infty)$ the following holds: 

If $f\in C^0_{\ell, p}((0, T)\times \R^n\times \R^n)$ then $f\in C^\alpha_{\ell, p-\alpha}((\tau, T)\times \R^n\times \R^n)$ for any $\tau\in (0, T)$, and the following estimate holds for all  $0 \le \tau_1 < \tau_2 < T$ with $|\tau_2-\tau_1| \le 1$,
\begin{align*}
\Vert f \Vert_{C_{\ell,p-\alpha}^{\alpha}((\tau_2,T)\times \R^n \times \R^n)} \le C (\tau_2 - \tau_1)^{-\frac{\alpha}{2}} \Vert f \Vert_{C_{\ell,p}^0((\tau_1,T) \times \R^n \times \R^n)},
\end{align*}
where $C > 0$ depends only on $n,p, m_0,M_0,p_0,M_q,$ and $q$.
\end{lemma}

\begin{proof}
The proof goes in the same way as the proof of \autoref{lemma:Holder-estimate}, applying the H\"older regularity estimate from \cite{GIMV19} to $\tilde{f}$ for any $z_0$. This is possible since $\tilde{f}$ is a solution to \eqref{eq:Landau-transformed} in $Q_r$, where $r := (\tau_2 - \tau_1)^{\frac{1}{2}}$, and because of \autoref{lemma:Landau-ellipticity} and \autoref{lemma:Landau-upper-ellipticity}.
We obtain for any $z_1,z_2 \in Q_{r/2}$:
\begin{align*}
|\tilde{f}(z_1) - \tilde{f}(z_2)| &\le C r^{-\alpha} (\Vert \tilde{f} \Vert_{L^{\infty}(Q_r} + r^2 \Vert \tilde{c} \tilde{f} \Vert_{L^{\infty}(Q_r)}) d_{\ell}(z_1,z_2)^{\alpha}.
\end{align*}
Undoing the change of variables, choosing $z_1 = 0$ implies that for any $\tilde{z}_2 \in Q_{r/2}(z_0)$:
\begin{align*}
|f(z_0) - f(\tilde{z}_2)| = |\tilde{f}(0) - \tilde{f}(z_2)| \le C r^{-\alpha} (1 + |v_0|)^{-p} \Vert f \Vert_{C^0_{\ell,p}((\tau_1,T) \times \R^n \times \R^n)} d_{\ell}(z_0,\tilde{z}_2)^{\alpha},
\end{align*}
where we also used that by \autoref{lemma:Landau-upper-ellipticity}
\begin{align*}
\Vert \tilde{c} \tilde{f} \Vert_{L^{\infty}(Q_r)} \le C  (1 + |v_0|)^{-2} \Vert f \Vert_{L^{\infty}(\mathcal{E}_r(z_0))} \le C (1 + |v_0|)^{-p-2} \Vert f \Vert_{C^0_{\ell,p}((\tau_1,T) \times \R^n \times \R^n)}.
\end{align*}
This concludes the proof by the same considerations as in \autoref{lemma:Holder-estimate}.
\end{proof}

We can finally conclude the proof of \autoref{cor:Landau}:

\begin{proof}[Proof of \autoref{cor:Landau}]

Thanks to  \autoref{lemma:Holder-estimate-Landau}, we can proceed in the exact same way as for the Boltzmann equation (setting $s = 1$ everywhere) and deduce an analog of \autoref{thm:Holder-estimate}. In particular, the first part of the claim holds. Moreover, when $f$ satisfies \eqref{eq:moment} for all $q > n$, we have that for any $\tau > 0$ and $p > 0$  
\begin{align}
\label{eq:Landau-Linfty}
\big\|(1+|v|)^p f\big\|_{L^{\infty}([\tau,T] \times \R^n \times \R^n)} \leq C_p 
\end{align}
with $C_p$ depending only on $n,m_0,M_0,p_0,p,\tau$, and on all $M_q$ for $q > n$.

To establish a higher order version of this estimate, we can proceed in the exact same way as in \cite{HeSn20}. Indeed, \cite[Proposition 3.2, Lemma 3.3]{HeSn20} remain true for $\gamma > 0$ without any change. These ingredients, together with, \autoref{lemma:Landau-ellipticity}, and \autoref{lemma:Landau-upper-ellipticity} allow us to apply the Schauder estimates from \cite{HeSn20} in an iterative way, in analogy to the proof of \cite[Theorem 1.2]{HeSn20}. Instead of the Gaussian decay in \cite[(3.6)]{HeSn20}, it suffices to use \eqref{eq:Landau-Linfty} for suitably large $p$.
\end{proof}

\section*{Conflict of interest statement}

The authors have no competing interests to declare that are relevant to the content of this article.

\section*{Data availability statement}

Data sharing is not applicable to this article as no datasets were generated or analysed.

\end{document}